\documentclass[12pt]{amsart}
\usepackage{graphicx}
\usepackage{amscd,amssymb,amsthm,amsmath,amssymb,mathrsfs,enumerate,amsfonts}

\usepackage{tikz}
\usepackage{tikz-cd}

\usepackage[margin=2cm]{geometry}

\usepackage[style=alphabetic]{biblatex} 
\newcommand{\F}{\mathbb{F}}
\renewcommand{\P}{\mathbb{P}}
\newcommand{\Z}{\mathbb{Z}}
\renewcommand{\O}{\mathcal{O}}
\newcommand{\C}{\mathbb{C}}

\newcommand{\A}{\mathbb{A}}

\newcommand{\Q}{\mathbb{Q}}
\newcommand{\wt}[1]{\widetilde{#1}}

\newcommand{\e}{\varepsilon}

\newtheorem{theorem}{Theorem}
\newtheorem{lemma}[theorem]{Lemma}
\newtheorem{cor}[theorem]{Corollary}

\theoremstyle{remark}
\newtheorem{remark}[equation]{Remark}

\theoremstyle{definition}
\newtheorem{definition}[theorem]{Definition}

\title{K-moduli of Fano threefolds of rank 2
and degree 28}
\author{Joseph Malbon}

\begin{document}

\begin{abstract}
    We find all possible K-polystable degenerations of smooth K-polystable Fano threefolds in the family \textnumero2.21, and describe the component of the moduli space $M^\text{K-ps}_{3,28}$ which parametrises this family.
\end{abstract}

\maketitle

\section{Introduction}

By now, K-stability has proved to be a very useful stability condition for Fano varieties as it enables the construction of projective moduli spaces of K-polystable Fano varieties, known as \emph{K-moduli spaces}. Considerable work has been devoted to explicitly describing these spaces and the objects they parametrise. For instance, in \cite{CFFK, CT, EJP, LX, LZ} the K-moduli components of families \textnumero1.13, \textnumero2.15, \textnumero3.3, \textnumero3.10, and \textnumero4.1 are described, respectively.

In \cite{book} it is shown that a general member of family \textnumero2.21 is K-stable, which implies that the corresponding K-moduli component $M^{\text{\textnumero} 2.21}$ is nonempty; moreover, it is a projective surface. The goal of this paper (cf. Theorem \ref{theorem:main}) is to describe all K-polystable Fano threefolds which arise as degenerations of members of family \textnumero2.21, and to give a complete description of the surface $M^{\text{\textnumero} 2.21}$.

\begin{definition}
    Fano threefolds in the family \textnumero2.21 are obtained by blowing up a smooth quadric hypersurface along a rational normal quartic curve. By the classification of smooth Fano threefolds in \cite{MM}, they are characterised by the invariants
    $$(-K_X)^3=28,\quad \rho(X)=2.$$  
    Let $M^{\text{K-ps}}_{3,28}$ be the moduli space of K-polystable $\Q$-Fano threefolds of volume $28$, whose existence is the celebrated result attributed to many people (cf. \cite{ABHLX20,BHLLX21,BLX19,BX19,CP21,Jia20,LWX21,LXZ22,Xu20,XZ20,XZ21}). Let $M^{\text{\textnumero} 2.21}$ denote the irreducible component of $M^{\text{K-ps}}_{3,28}$ whose general point parametrises a smooth member of family \textnumero2.21, and let $M^\text{sm}\subset M^\text{\textnumero 2.21}$ be the subset parametrising smooth members.
\end{definition}

In \cite{K-STAB} it was shown that for members of the family \textnumero2.21, K-polystability is equivalent to GIT-polystability for the action of $\mathrm{PGL}_2(\C)$ on the natural parameter space:
\begin{equation}
    \label{eq:parameter space}
    \P\bigg(H^0\big(\P^4, \mathcal{I}_{C_4}(2)\big)\bigg)
\end{equation}
It follows that the moduli space $M^{\mathrm{sm}}$ is isomorphic to an open subset of the corresponding GIT-quotient. This quotient gives a modular compactification of $M^{\mathrm{sm}}$ which, whilst being projective, parametrises varieties which are not K-polystable. We thus give an alternative compactification of $M^{\mathrm{sm}}$ which exclusively parametrises K-polystable degenerations of family \textnumero2.21. Namely, we show the following:
\begin{theorem}
\label{theorem:main}
    The boundary $\partial M^\text{\textnumero2.21}:= M^\text{\textnumero 2.21}\setminus M^\text{sm}$ is a divisor on $M^\text{\textnumero 2.21}$ which has three components $\Delta_1,\Delta_2,\Delta_3$, each isomorphic to $\P^1$, with $\Delta_1\cap\Delta_3=\varnothing$ and $\Delta_2$ meeting $\Delta_1$ and $\Delta_3$ each in one point. Moreover, the boundary admits the following stratification: 
    $$\partial M^\text{\textnumero2.21}=\Delta_1^\circ\sqcup\Delta_{12}\sqcup \Delta_2^\circ\sqcup\Delta_{23}\sqcup\Delta_3^\circ,$$
    where:
    \begin{align*}
        \Delta_{12}&=\Delta_1\cap\Delta_2,\\ 
        \Delta_{23}&=\Delta_2\cap\Delta_3 ,\\ 
        \intertext{and:} \\
        \Delta_1^\circ&=\Delta_1\setminus\Delta_{12},\\
        \Delta_2^\circ&=\Delta_2\setminus(\Delta_{12}\cup\Delta_{23}),\\
        \Delta_3^\circ&=\Delta_3\setminus \Delta_{23}.
    \end{align*}   
    These strata have the following modular interpretations:
    \begin{itemize}
        \item $\Delta_1^\circ$ parametrises threefolds of the form $\mathrm{Bl}_{C_4}Q^\mathrm{iso}$, where $Q^\mathrm{iso}$ is a quadric threefold with an isolated singularity; \\
        \item $\Delta_2^\circ$ parametrises threefolds of the form $\mathrm{Bl}_{C_{1,2,1}}Q^\mathrm{sm}$, where $Q^\mathrm{sm}$ is a smooth quadric threefold and $C_{1,2,1}$ is a union of two disjoint lines and a conic, with each line meeting the conic transversely in one point: 
        
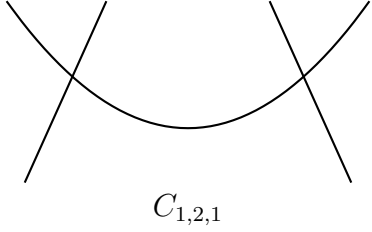
\begin{figure}[h]
\centering
\begin{tikzpicture}[scale=1.2]
  \draw[thick, domain=-2:2, smooth, variable=\x]
    plot ({\x},{0.35*\x*\x});

  \draw[thick] (-1.8,-0.6) -- (-0.9,1.4);

  \draw[thick] (1.8,-0.6) -- (0.9,1.4);

  \node at (0,-0.9) {$C_{1,2,1}$};
\end{tikzpicture}
\end{figure}
        \item $\Delta_3^\circ$ parametrises threefolds of the form $\mathrm{Bl}_{C_{2,2}}Q^\mathrm{sing}$, where for every point of $\Delta_3^\circ$ except one, we have $Q^\mathrm{sing}\cong Q^\mathrm{iso}$, and $C_{2,2}$ is a union of two conics meeting transversely in a single point:

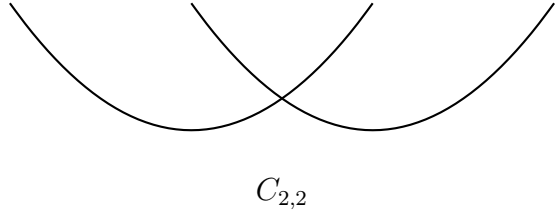
\begin{figure}[h]
\centering
\begin{tikzpicture}[scale=1.2]
\draw[thick, domain=-1:3, smooth, variable=\x]
    plot ({\x},{0.35*(\x-1)^2});
\draw[thick, domain=-3:1, smooth, variable=\x]
    plot ({\x},{0.35*(\x+1)^2});
    \node at (0,-0.7) {$C_{2,2}$};
\end{tikzpicture}
\end{figure}
    and at the special point, $Q^\mathrm{sing}$ is isomorphic to an irreducible quadric with a non-isolated singularity $Q^\mathrm{niso}$, and the curve $C_{2,2}^\mathrm{nred}$ is a non-reduced scheme supported on a conic $C_{2,2}^\mathrm{red}$. Moreover, there exists a cubic ruled surface $S\cong\F_1\subset \P^4$ containing $C_{2,2}^\mathrm{nred}$ such that, in $\mathrm{Div}(S)$, $C_{2,2}^\mathrm{nred}=2C_{2,2}^\mathrm{red}$:
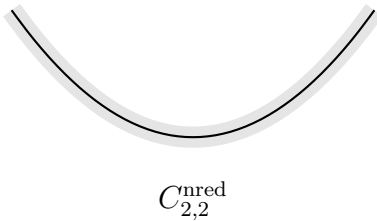
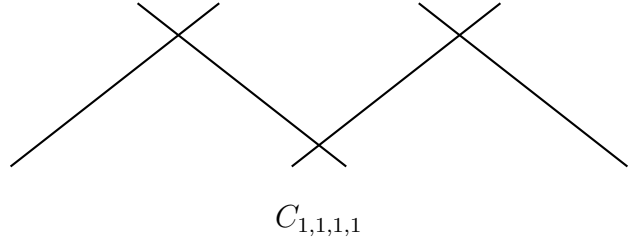
\begin{figure}[h]
\centering
\begin{tikzpicture}[scale=1.2]
  \draw[line width=8pt, gray, opacity=0.2, domain=-2:2, smooth, variable=\x]
    plot ({\x},{0.35*\x*\x});

  \draw[thick, black, domain=-2:2, smooth, variable=\x]
    plot ({\x},{0.35*\x*\x});

    \node at (0,-0.7) {$C_{2,2}^\mathrm{nred}$};
\end{tikzpicture}
\end{figure}
    \item $\Delta_{12}$ parametrises the unique K-polystable Fano threefold $\mathrm{Bl}_{C_{1,2,1}}Q^\mathrm{iso}$;\\
    \item $\Delta_{23}$ parametrises the unique K-polystable Fano threefold $\mathrm{Bl}_{C_{1,1,1,1}}Q^\mathrm{iso}$, where $C_{1,1,1,1}$ is a union of 4 lines meeting according to the $A_4$-Dynkin diagram:
\begin{figure}
\begin{tikzpicture}[scale=1.2]
  \draw[thick] (-3.4,-0.4) -- (-1.1,1.4);
  \draw[thick] (-2.0,1.4) -- (0.3,-0.4);
  \draw[thick] (-0.3,-0.4) -- (2.0,1.4);
  \draw[thick] (1.1,1.4) -- (3.4,-0.4);

  \node at (0,-1.0) {$C_{1,1,1,1}$};
\end{tikzpicture} 
\end{figure}
\end{itemize}
\vspace{0.5cm}
In each of the above cases, the singularities of the quadric are disjoint from the quartic curve, and moreover this curve is not contained in a hyperplane in $\P^4$.
\end{theorem}
The following corollary is immediate:
\begin{cor}
    Let $X$ be a $\Q$-Fano threefold which is K-polystable and $\Q$-Gorenstein smoothable to a smooth Fano threefold in \textnumero2.21. Then $X$ is the blow-up of a quadric hypersurface along a connected curve in $\P^4$ of degree 4 and arithmetic genus 0.
\end{cor}

\noindent\textbf{Structure of the paper}. The rest of this paper is dedicated to the proof of Theorem \ref{theorem:main}. Let us give an overview of the proof and a description of the paper. In Section \ref{section:families} we find an open cover $T=\bigcup U_{ij}$, where $T$ is the blow-up of $\P^2$ in 6 points (not in general position), and flat projective morphisms:
$$\big\{\mathscr{X}_{ij}\to U_{ij}\big\}_{i=1,\ldots,6,\ j=1,2}$$
such that:
\begin{enumerate}
    \item For every smooth member $X$ of family \textnumero2.21, there exists $i,j$, and $x\in U_{ij}$ such that the corresponding fibre of $\mathscr{X}_{ij}$ is isomorphic to $X$,
    \item For each $i,j$, $K_{\mathscr{X}_{ij}/ U_{ij}}$ is Cartier,
    \item For each $i_1,j_1$ and $i_2,j_2$, and every $x\in U_{i_1j_1}\cap U_{i_2j_2}$, the fibres of $\mathscr{X}_{i_1j_1}$ and $\mathscr{X}_{i_2j_2}$ over $x$ are isomorphic.
\end{enumerate}

In Section \ref{section:K-polystability}, we show that every fibre of each of the above families is a K-polystable Fano variety with $\big(-K_{(\mathscr{X}_{ij})_x}\big)^3=28$. We prove this in a case-by-case manner, making use of the Abban--Zhuang method of admissible flags to show that for each fibre $X$, the inequality:
$$A_X(D):=1+\mathrm{discrep}_X(D)>S(-K_X;D):=\frac{1}{{\mathrm{vol}(-K_X)}}\int_0^\infty\mathrm{vol}(-K_X-uD)du$$
holds for every $G$-invariant prime divisor $D$ over $X$ for a suitably chosen subgroup $G\subset\mathrm{Aut}(X)$, thereby concluding that $X$ is K-polystable by the Fujita--Li valuative criterion (\cite{Fujita2019}) and its equivariant version (\cite{ZZ}).

This, in conjunction with properties $(1)$ and $(2)$ as above, yields morphisms of stacks:
$$U_{ij}\to\mathscr{M}^\text{\textnumero2.21}$$
for each $i,j$, where $\mathscr{M}^\text{\textnumero2.21}$ denotes the irreducible component of the moduli stack $\mathscr{M}^\text{K-ss}_{3,28}$ (see \cite[Theorem 2.13]{LZ} for the definition) whose general point parametrises a K-stable Fano threefold in family \textnumero2.21. These morphisms descend to morphisms of good moduli spaces:
$$U_{ij}\to M^\text{\textnumero2.21},$$
which by property $(3)$ agree on overlaps, and hence glue to give a morphism
$$T\to M^\text{\textnumero2.21}.$$
Furthermore, we also find an action of a finite group $\Gamma$ on $T$ whose orbits give the fibre-isomorphism equivalence relation for each family $\mathscr{X}_{ij}\to U_{ij}$, thereby obtaining a bijective morphism:
$$T/\Gamma\to M^\text{\textnumero2.21}.$$
Now, to conclude that this morphism is an isomorphism, we must show that $M^{\text{K-ps}}_{3,28}$ is normal at the moduli point of every fibre of the families $\{\mathscr{X}_{ij}\to U_{ij}\}$. If $X$ is a fibre with at worst isolated ordinary double points (in particular, terminal Gorenstein) then by \cite[Proposition 3]{Namikawa} the deformations of $X$ are unobstructed. This implies that the stack $\mathscr{M}^\text{\textnumero2.21}$ is smooth at the moduli point $[X]$, and then it follows from \cite[Theorem 4.16(viii)]{Alper} that the K-moduli space $M^\text{\textnumero2.21}$ is normal at $[X]$.

The only fibre of $\{\mathscr{X}_{ij}\to U_{ij}\}$ with non-terminal singularities is, up to isomorphism, the variety $\mathrm{Bl}_{C_{2,2}^\mathrm{nred}}(Q^\mathrm{niso})$. In Section \ref{section:deformation} we show that its deformations are unobstructed, which completes the proof that $M^\text{\textnumero2.21}$ is normal, and hence that it is isomorphic to the quotient $T/\Gamma$. Under this correspondence, the divisors $\Delta_1,\Delta_2,\Delta_3$ described in the statement of Theorem \ref{theorem:main} are given by images of the lines on $T$ with negative self-intersection under the quotient map. The moduli points of these divisors correspond to the fibres of the $\mathscr{X}_{ij}$ over these lines.\\ 

\noindent\textbf{Acknowledgements.}
The author thanks Prof. Ivan Cheltsov for his guidance during this project, and also the Simons Foundation for partially funding the research travel that made this work possible. 

\section{The families $\{\mathscr{X}_{ij}\to U_{ij}\}$, and the group $\Gamma$}
\label{section:families}

In this section, we construct the open cover $T=\bigcup U_{ij}$ and morphisms $\mathscr{X}_{ij}\to U_{ij}$ satisfying properties $(1)-(3)$ in the previous section.

First, observe that the boundary threefolds described in
Theorem \ref{theorem:main} are obtained by blowing up reducible
centres. This suggests that the universal family over the natural
parameter space defined in \eqref{eq:parameter space} is not suitable
for constructing the desired compactification of the family, since
it does not parametrise degenerations of the curve $C_4$. With this
in mind, we now give a different parametrisation of the smooth
members.

Recall that the adjugate of an $(n\times n)$-matrix $A$ has the property that:
$$\mathrm{adj}(\mathrm{adj}(A))=\mathrm{det}(A)^{n-2}A,$$
and therefore defines a birational involution, $\chi$, on the projectivisation of the space of $(n\times n)$-matrices. Let us identify $\P^5$ with the projectivisation of the space of symmetric $(3\times3)$-matrices, and write $\chi$ as:
\begin{equation}
    \label{eq:involution}
    \begin{aligned}
    \chi\colon \P^5&\dashrightarrow\P^5 \\ 
    \begin{pmatrix}
x_0 & x_1 & x_5 \\
x_1 & x_2 & x_3 \\
x_5 & x_3 & x_4
\end{pmatrix}&\mapsto 
\begin{pmatrix}
g_0 & g_1 & g_5 \\
g_1 & g_2 & g_3 \\
g_5 & g_3 & g_4
\end{pmatrix}
\end{aligned}
\end{equation}
where 
\begin{align*}
    g_0 &= x_2 x_4 - x_3^2, \\
    g_1 &= x_3 x_5-x_1 x_4, \\
    g_2 &= x_0 x_4 - x_5^2, \\
    g_3 &= x_1 x_5-x_0 x_3, \\
    g_4 &= x_0 x_2 - x_1^2, \\
    g_5 &= x_1 x_3 - x_2 x_5.
\end{align*}
Note that for symmetric $(3\times3)$-matrices, the adjugate vanishes precisely on matrices of the form $\mathbf{v}\mathbf{v}^{T}$, for $\mathbf{v}\in \C^3$. Thus, the base locus of $\chi$ is equal to the standard Veronese surface $S\subset\P^5$:
$$[x:y:z]\mapsto[x^2:xy:y^2:yz:z^2:xz].$$
Let $\pi\colon\widetilde{\P}^5\to\P^5$ be the blow-up along $S$, and let $E$ be the $\pi$-exceptional divisor. Then $\chi$ lifts to a biregular involution $\tilde{\chi}$ of $\widetilde{\P}^5$, which by construction swaps the linear systems $|\pi^*\O_{\P^5}(1)|$ and $|\pi^*\O_{\P^5}(2) - E|$ on $\widetilde{\P}^5$. Moreover, if we write $E'\subset\wt{\P}^5$ for the $\pi$-strict transform of the secant variety of $S$, which itself is a cubic fourfold singular along $S$, then the involution $\tilde{\chi}$ swaps $E$ and $E'$. 

Now, since $\tilde{\chi}$ is an involution we have that for any $H\in|\pi^*\O_{\P^5}(1)|$, the intersection 
$$X:=H\cap\tilde{\chi}(H)$$
is $\tilde{\chi}$-invariant. If $H$ is general, then this threefold $X$ is equal to the blow-up of a quadric threefold $\pi(\tilde{\chi}(H))\cap\pi(H)$ along the rational normal quartic curve $S\cap\pi(H)$, hence is a member of the family \textnumero2.21.

With this background established, let us now start defining our families. Let $$\mathscr{H}=\big\{ax_0+bx_2+cx_4=0\big\}\subset\wt{\P}^5\times \P^2,$$
where $a,b,c$ are coordinates on this $\P^2$ and by slight abuse of notation we view the $x_i$ as coordinates on $\wt{\P}^5$. The involution $\tilde{\chi}$ on $\wt{\P}^5$ extends naturally to $\wt{\P}^5\times \P^2$ by acting trivially on $\P^2$; let
$$\mathscr{X}=\mathscr{H}\cap\tilde{\chi}(\mathscr{H}),$$
which we view as a family of threefolds $\mathscr{X}\to \P^2$. By construction, every fibre of this family is of the form $\mathrm{Bl}_C(Q)$, where $Q:=\pi(H)\cap\pi(H')$ is quadric threefold in $\pi(H)\cong\P^4$ and $C:=\pi(H)\cap S\subset Q$ is a (not necessarily integral) curve of degree four. One checks that this morphism is smooth precisely over $\P^2\setminus(\Delta_1\cup\Delta_3)$, where:
\begin{align*}
    \Delta_1&:=\{(a + b + c)(a - b - c)(a - b + c)(a + b - c)=0\},\\
    \Delta_3&:=\{abc=0\}.
\end{align*}

\begin{remark}
This is an abuse of notation since in Theorem \ref{theorem:main} we have already defined the $\Delta_i$ to be divisorial components of the boundary of $M^\text{\textnumero2.21}$. However, under the quotient by the $\Gamma$-action the components of $\Delta_1$ and $\Delta_3$ (and $\Delta_2$ which we define in the next subsection) are identified with the corresponding 
boundary components; we hope this vindicates the notation. 
\end{remark}

\begin{lemma}
\label{lemma:Fanoness}
Let $\Sigma:=\Delta_1\cap\Delta_3$. Then for every $x\in \P^2\setminus \Sigma$, the fibre $\mathscr{X}_x$ of $\mathscr{X}$ over $x$ is a Fano threefold of degree 28 with at worst ordinary double point singularities.
\end{lemma}

\begin{proof}
Let us first prove that $X:=\mathscr{X}_x$ is a Fano variety of degree 28. By the blow-up formula we have that $-K_{\wt{\P}^5}\sim 6H-2E=2H+2H'$, where $H'\in |\pi^*\O_{\P^5}(2) - E|$. For $x\in\P^2\setminus \Sigma $ we have that $X\subset\wt{\P}^5$ is an integral complete intersection subscheme of the form $H\cap H'$, so that its normal bundle is:
$$\mathcal{N}_{X/\wt{\P}^5}\cong \O_X(H|_X)\oplus \O_X(H'|_X).$$
Hence by the adjunction formula we have that $-K_X\sim H|_X+H'|_X$, which is ample because it is the restriction of the very ample line bundle $\O_{\P^5\times\P^5}(1,1)$ to $X\subset\P^5\times\P^5$. Moreover since $(-K_X)^3$ is invariant under deformation, we have that $(-K_X)^3=28$.

Let us now examine the singular fibres. If $x\in \Delta_1\setminus \Sigma$ the corresponding fibre is a threefold of the form $\mathrm{Bl}_{C_4}(Q^\mathrm{iso})$, which has an isolated ordinary double point since $\mathrm{Sing}(Q^\mathrm{iso})\cap C_4=\varnothing$ in this case. If $x\in \Delta_3\setminus \Sigma$ the corresponding fibre is a threefold of the form $\mathrm{Bl}_{C_{2,2}}(Q^\mathrm{iso})$ or $\mathrm{Bl}_{C_{2,2}^\mathrm{nred}}(Q^\mathrm{niso})$. Threefolds of the form $\mathrm{Bl}_{C_{2,2}}(Q^\mathrm{iso})$ have two isolated ordinary double points since $\mathrm{Sing}(Q^\mathrm{iso})\cap C_{2,2}=\varnothing$ in this case.

The threefold $\mathrm{Bl}_{C_{2,2}^\mathrm{nred}}(Q^\mathrm{niso})$ has singular locus consisting of two disjoint curves $Z_1\cup Z_2$. One, say $Z_1,$ comes from the singularities of $Q^\mathrm{niso}$ which is a line of transverse $A_1$-singularities. The other, $Z_2$, is contained in the exceptional divisor of the blow-up of $Q^\mathrm{niso}$ along $C_{2,2}^\mathrm{nred}$. Since $Q^\mathrm{niso}$ is smooth along $C_{2,2}^\mathrm{nred}$, then the curve $C_{2,2}^\mathrm{nred}$ is locally of the form $V(x^2,y)$ inside of $\C_{x,y,z}^3$. Therefore $\mathrm{Bl}_{C_{2,2}^\mathrm{nred}}(Q^\mathrm{niso})$ is given locally by $\{xz-y^2=0\}\subset\C^4_{x,y,z,w}$, with $Z_2=\{x=y=z=0\}$, and hence $Z_2$ is also a smooth curve of transverse $A_1$-singularities.
\end{proof}

\subsection{The group action}
\label{subsec:group action}
We have a natural action of the group $\mathrm{PGL}_3(\C)$ on this $\P^5$ given by:
\begin{align*}
    \mathrm{PGL}_3(\C) \times \P^5&\to \P^5 \\ 
    ([M],[A])&\mapsto [M.A.M^\mathrm{T}].
\end{align*}
Clearly this action preserves the locus of matrices of the form $\mathbf{v}\mathbf{v}^{\mathrm{T}}$, for $\mathbf{v}\in \C^3$, and hence preserves the Veronese surface $S$. This $\mathrm{PGL}_3(\C)$-action therefore lifts to an action on $\widetilde{\P}^5$. The following properties are straightforward to prove:
\begin{itemize}
    \item The injective homomorphism $\mathrm{PGL}_3(\C)\hookrightarrow \mathrm{Aut}(\P^5, S)$ is an isomorphism, where this latter group is the group of projective transformations of $\P^5$ which preserve $S$. Moreover, $\mathrm{Aut}(\wt{\P}^5)$ is generated by $\mathrm{Aut}(\P^5, S)$ and $\tilde{\chi}$,
    \item $\mathrm{PGL}_3(\C)$ and $\tilde{\chi}$ satisfy the identity:
$$\mathrm{adj}(M. A. M^\mathrm{T}) =  \mathrm{det}(M)^2 . (M^{-1})^\mathrm{T}  . \mathrm{adj}(A). M^{-1}.$$
\end{itemize}


\begin{definition}
\label{def:H}
Let $\Gamma\subset\mathrm{PGL}_3(\C)$ be the subgroup, isomorphic to $\Z_4^2\rtimes\mathfrak{S}_3$, generated by the matrices:
$$
M_1:=\begin{pmatrix}
    1 & 0 & 0 \\ 
    0 & i  & 0 \\
    0 & 0 & 1
\end{pmatrix},\quad
M_2:=\begin{pmatrix}
    0&0&1 \\ 
    0&1 & 0 \\
    1 & 0& 0
\end{pmatrix},\quad
M_3:=\begin{pmatrix}
    0 & 0 & 1 \\ 
    1& 0 & 0 \\ 
    0 & 1 & 0
\end{pmatrix}.
$$
We consider $\Gamma$ to be acting on $\wt{\P}^5$ as above. There is also a $\Gamma$-action on the base $\P^2$ of the family $\mathscr{X}$ such that the morphism $\mathscr{X}\to \P^2 $ is $\Gamma$-equivariant, given by:
\begin{align*}
\Gamma&\longrightarrow \mathrm{PGL}_3(\C), \\ 
M_1, M_2, M_3&\longmapsto M_1^2,  M_2,  M_3.
\end{align*}
We note that the image of this projective representation is isomorphic to $\Z_2^2\rtimes\mathfrak{S}_3\cong \mathfrak{S}_4$.
\end{definition}

\begin{lemma}
\label{lemma:orbits}
    For all $x,y\in\P^2\setminus\Sigma$, we have that $\mathscr{X}_x\cong\mathscr{X}_y$ if and only if $y\in \Gamma\cdot x$.
\end{lemma}
\begin{proof}
    Let us prove the forward implication; the converse is true by the construction of $\mathscr{X}\to\P^2$.
    
    We first prove that every isomorphism $\mathscr{X}_x\cong\mathscr{X}_y$ is induced by an element of the group $\mathrm{Aut}(\P^5, S)$ which maps the corresponding quadric threefolds $\pi(\mathscr{X}_x),\pi(\mathscr{X}_y)$ to one another. Since every fibre $X$ of $\mathscr{X}\to \P^2$ is the blowup of a quadric threefold along a connected lci subscheme of codimension $\geq 2$ we have that $\mathrm{Pic}(X)\cong \Z[H|_X]\oplus \Z[E|_X]$. Moreover, we have that $\mathrm{Nef}(X)$ is spanned by $\mathbb{R}_{\geq0}[H|_X]$ and $\mathbb{R}_{\geq0}[H'|_X]$, and that these two rays are swapped under the action of $\tilde{\chi}|_X$ on $\mathrm{Nef}(X)$.
    
    Suppose that we have an isomorphism $\varphi\colon \mathscr{X}_x\xrightarrow{\sim} \mathscr{X}_y$. Then $\varphi$ maps the extremal rays of $\mathrm{Nef}(\mathscr{X}_x)$ to the extremal rays of $\mathrm{Nef}(\mathscr{X}_y)$, so that we either have that:
    $$
    \varphi^*H|_{\mathscr{X}_y}=H|_{\mathscr{X}_x}\text{ and }    \varphi^*H'|_{\mathscr{X}_y}=H'|_{\mathscr{X}_x}
    $$
    or
    $$
    \varphi^*H|_{\mathscr{X}_y}=H'|_{\mathscr{X}_x}\text{ and }\varphi^*H'|_{\mathscr{X}_y}=H|_{\mathscr{X}_x}.
    $$
    Let us suppose we are in the first case. Then $\varphi$ induces an isomorphism $\overline{\varphi}\colon\Pi_x\xrightarrow{\sim}\Pi_y$ between the hyperplanes $\Pi_\alpha\subset\P^5$ spanned by $\pi(\mathscr{X}_\alpha)$, for $\alpha\in\{x,y\}$, such that the following diagram commutes:
    $$\begin{tikzcd}
	{\mathscr{X}_x} & {\mathscr{X}_y} \\
	{\Pi_x} & {\Pi_y}
	\arrow["\varphi", from=1-1, to=1-2]
	\arrow["\pi|_{\mathscr{X}_x}"', from=1-1, to=2-1]
	\arrow["\pi|_{\mathscr{X}_y}", from=1-2, to=2-2]
	\arrow["\overline{\varphi}", from=2-1, to=2-2]
\end{tikzcd}$$
The morphism $\varphi$ maps the exceptional divisor $E|_{\mathscr{X}_x}$ of $\pi|_{\mathscr{X}_x}$ to that of $\pi|_{\mathscr{X}_y}$, so that $\overline{\varphi}$ maps the curve $ C_x:=\Pi_x\cap S$ to $C_y:=\Pi_y\cap S$. The induced map on Picard groups:
$$\overline{\varphi}|_{C_x}^*\colon\mathrm{Pic}(C_y)\xrightarrow{\sim}\mathrm{Pic}(C_x)$$
maps $\O_S(2)|_{C_y}$ to $\O_S(2)|_{C_x}$ (under the identification $S\cong\P^2)$, which implies that $\overline{\varphi}|_{C_x}^*$ maps $\O_S(1)|_{C_y}$ to $\O_S(1)|_{C_x}$ since the Picard group of a plane conic has no 2-torsion. Since $C_\alpha$ is linearly normal in $\P^2$, the map $\overline{\varphi}|_{C_x}$ lifts to a unique automorphism $\psi\in\mathrm{Aut}(S)\cong\mathrm{PGL}_3(\C)$ such that $\psi|_{C_x}=\overline{\varphi}|_{C_x}$.

We consider this $\psi$ to be acting on $\P^5$ as described in (\ref{subsec:group action}). Since each $C_\alpha$ is nondegenerate and linearly normal in $\Pi_\alpha$, we have that $\psi$ maps $\Pi_x\xrightarrow{\sim}\Pi_y$ and moreover $\psi|_{\Pi_x}$ is uniquely determined by its restriction to $C_x$. Therefore, we have that $\psi|_{\Pi_x}=\overline{\varphi}$. The map $\psi$ lifts to an automorphism $\tilde{\psi}$ of $\wt{\P}^5$ whose restriction to $\mathscr{X}_x$ agrees with $\varphi$ on a dense open subset of $\mathscr{X}_x$, and hence $\varphi=\tilde{\psi}|_{\mathscr{X}_x}$. Therefore $\mathscr{X}_x$ and $\mathscr{X}_y$ are related by an element of $\mathrm{Aut}(\P^5,S)$. A computer calculation then shows that this element belongs to the subgroup $\Gamma$ of $\mathrm{Aut}(\P^5,S)\cong\mathrm{PGL}_3(\C)$.

Now suppose that $\varphi^*H|_{\mathscr{X}_y}=H'|_{\mathscr{X}_x}$. Then $(\varphi\circ\tilde{\chi}|_{\mathscr{X}_x})^*H|_{\mathscr{X}_y}=H|_{\mathscr{X}_x}$, so that by the same argument as above we have that there exists $\psi\in\mathrm{Aut}(\P^5, S)$ whose lift to $\wt{\P}^5$ satisfies $\tilde{\psi}|_{\mathscr{X}_x}=\varphi\circ\tilde{\chi}|_{\mathscr{X}_x}$, so that $\mathscr{X}_x$ and $\mathscr{X}_y$ are again related by an element of $\mathrm{Aut}(\P^5, S)$, and hence an element of $\Gamma$.
\end{proof}

\subsection{Modification to the family}
In Section \ref{section:K-polystability} it is shown that, for $p\in \P^2\setminus\Sigma$, the fibre of $\mathscr{X}$ over $p$ is K-polstable. Now, observe that the fibres of $\mathscr{X}$ over $\Sigma$ are reducible. By \cite[Theorem 1.2]{Odaka}, K-semistable Fano varieties have klt singularities, and hence are irreducible. Since we are trying to build families of K-polystable (and therefore K-semistable) Fano varieties, we must make a modification to the family $\mathscr{X}$ in order to obtain one with only K-polystable fibres. \\

To this end, let $\varepsilon\colon T\to\P^2$ be the blow-up along $\Sigma$, and let $\Delta_2\subset T$ be the $\varepsilon$-exceptional divisor. Let $U_0:=T\setminus\Delta_2$, and let $\mathscr{X}_0\to U_0$ denote the restriction to $U_0$ of the $\varepsilon$-pullback $\mathscr{X}\times_{\P^2}T$; every fibre of this family is irreducible. Now, for each point $p_i$ in the set
$$\Sigma=\Big\{\,[1:0:-1], \,\,[1:0:1], \,\,[1:1:0], \,\,[1:-1:0],\,\,[0:1:1],\,\,[0:1:-1]\,\Big\}$$
we define open subsets $U_i\subset T$ by:
$$U_i:=\varepsilon^{-1}\big(\P^2\setminus (\Sigma_i\cup\check{p}_i)\big),$$
where $\check{p}_i\subset\P^2$ is the line which is projective dual to $p_i$ and $\Sigma_i:=\Sigma\setminus\{p_i\}$. Let $E_i$ be the component of $\Delta_2$ over $p_i$. Then $U_i$ is an open neighbourhood of $E_i$, and is disjoint from the other exceptional divisors. Writing $U_{i1}$ and $U_{i2}$ for the natural affine charts of $U_i$ then as $p_i$ ranges over $\Sigma$ we thus obtain an open cover:
\begin{equation}
\label{eq:cover}
T=\bigcup_i \big(U_{i1}\cup U_{i2}\big).
\end{equation}
We do not include the open set $U_0$ in this cover, since it is redundant, although we will use the family $\mathscr{X}_0$ to obtain isomorphisms between the fibres of the other families.

Let us now define a family of threefolds over each open set in the cover (\ref{eq:cover}). We introduce a \emph{deformed} version of the adjugate for any $t\in \C$, which on symmetric $(3\times 3)$-matrices is given by:
\begin{equation}
    \begin{aligned}
    \chi_t\colon \P^5&\dashrightarrow\P^5 \\ 
    \begin{pmatrix}
x_0 & x_1 & x_5 \\
x_1 & x_2 & x_3 \\
x_5 & x_3 & x_4
\end{pmatrix}&\mapsto 
\begin{pmatrix}
g_0(t) & g_1(t) & g_5(t) \\
g_1(t) & g_2(t) & g_3(t) \\
g_5(t) & g_3(t) & g_4(t)
\end{pmatrix},
\end{aligned}
\end{equation}
where 
\begin{align*}
    g_0(t) &= x_2 x_4 - tx_3^2, \\
    g_1(t) &= tx_3 x_5-x_1 x_4, \\
    g_2(t) &= x_0 x_4 - t^2x_5^2, \\
    g_3(t) &= tx_1 x_5-x_0 x_3, \\
    g_4(t) &= x_0 x_2 - tx_1^2, \\
    g_5(t) &= x_1 x_3 - x_2 x_5.
\end{align*}
One checks that this deformed adjugate is a birational involution for every $t\in\C$. Then each open set $U_{ij}$ in the cover (\ref{eq:cover}), after choosing affine coordinates $(t_{ij},v_{ij})$ such that $E_i\cap U_{ij}$ is defined by $t_{ij}=0$, parametrises the involutions $\chi_t$ as follows:
\begin{align*}
    \psi_{ij}\colon\P^5\times U_{ij}&\dashrightarrow\P^5\times U_{ij} \\ 
    \big(p, t_{ij}, v_{ij}\big) &\mapsto \big(\chi_{t_{ij}}(p), t_{ij},v_{ij}).
\end{align*}
The base locus of each $U_{ij}$-birational map $\psi_{ij}$ is therefore given by:
$$\mathscr{S}_{ij}=\big\{g_0(t_{ij})=g_1(t_{ij})=g_2(t_{ij})=g_3(t_{ij})=g_4(t_{ij})=g_5(t_{ij})=0\big\}.$$
Let us now define the families. Let 
\begin{align*}\mathscr{H}_{i1}&=\bigg\{\dfrac{1}{2}(x_0+x_4) + v_{i1}x_2 - x_5 = 0\bigg\}\subset\P^5\times U_{i1},\\
\mathscr{H}_{i2}&=\bigg\{\dfrac{v_{i2}}{2}(x_0+x_4) + x_2 - x_5 = 0\bigg\}\subset\P^5\times U_{i2},
\end{align*}
for $i=1,\ldots, 6$. Then each the restriction $\psi_{ij}$ to $\mathscr{H}_{ij}$ is birational map whose base locus is equal to $\mathscr{C}_{ij}:=\mathscr{S}_{ij}\cap \mathscr{H}_{ij}$.
\begin{remark}
    For each $v_{ij}$, the coordinate-projection morphism $\mathscr{C}_{ij}\to \A^1_{t_{ij}}$ is a one-parameter degeneration of rational normal quartics to a curve of type $C_{1,2,1}$ or $C_{1,1,1,1}$.
\end{remark}
Then
\begin{align*}\psi_{i1}(\mathscr{H}_{i1})&=\bigg\{\dfrac{1}{2}(g_0(t_{i1})+g_4(t_{i1})) + v_{i1}g_2(t_{i1}) - g_5(t_{i1}) = 0\bigg\}\subset\P^5\times U_{i1},\\
\psi_{i2}(\mathscr{H}_{i2})&=\bigg\{\dfrac{v_{i2}}{2}(g_0(t_{i2})+g_4(t_{i2})) + g_2(t_{i2}) - g_5(t_{i2}) = 0\bigg\}\subset\P^5\times U_{i2},
\end{align*}
Let $\mathscr{Q}_{ij}=\mathscr{H}_{ij}\cap \psi_{ij}(\mathscr{H}_{ij})$. Then $\mathscr{C}_{ij}\subset \mathscr{Q}_{ij} $.

\begin{definition}
    For each $i\in\{1,\ldots,6\},j\in\{1,2\}$, let $\mathscr{X}_{ij}$ be the blow-up of the family of quadrics $\mathscr{Q}_{ij}$ along the family of curves $\mathscr{C}_{ij}$. 
\end{definition}
\begin{remark}
    Each $\mathscr{Q}_{ij}$ is $\psi_{t_{ij}}$-invariant. Therefore $\psi_{ij}$ lifts to a biregular involution, $\tilde{\psi}_{ij}$ on $\mathscr{X}_{ij}$.
\end{remark}

\begin{lemma}
    Each family $\mathscr{X}_{ij}\to U_{ij}$ is a proper, flat, Gorenstein morphism whose fibres are Gorenstein Fano varieties isomorphic to $\mathrm{Bl}_{(C_{ij})_u}(\mathscr{Q}_{ij})_u$. 
\end{lemma}
\begin{proof}

Let us show that for each family $\mathscr{X}_{ij}\to U_{ij}$ and every $u\in U_{ij}$, we have that:
    $$(\mathscr{X}_{ij})_u\cong\mathrm{Bl}_{(C_{ij})_u}(\mathscr{Q}_{ij})_u.$$
    We claim that the inclusion $\mathscr{C}_{ij}\hookrightarrow \mathscr{Q}_{ij} $ is a regular embedding. Indeed, $\mathscr{C}_{ij}$ is covered by the affine opens $\{x_i\neq 0\}\subset \P^5\times U_{ij}$, for $i=0,\ldots,4$. It is straightforward to check by hand that on each of these affine patches, the 7 equations defining the three-dimensional variety $\mathscr{C}_{ij}$ inside of the five-dimensional variety $\mathscr{Q}_{ij}\cap\{x_i\neq0\}$ reduce to a regular sequence of length 2. 
    
    Now since $\mathscr{Q}_{ij}$ is a complete intersection inside the smooth variety $\P^5\times U_{ij}$, it follows that $\mathscr{C}_{ij}$ is an lci variety since the composition of lci morphisms is lci (cf. \cite[Tag 069J]{stacks-project}), and hence that it is Cohen-Macaulay. Moreover $U_{ij}$ is smooth, and the coordinate projection map $\mathscr{C}_{ij}\to U_{ij}$ is equidimensional, hence flat by Miracle Flatness (cf. \cite[Tag 00R4]{stacks-project}). Therefore
    $$(\mathscr{X}_{ij})_u\cong\mathrm{Bl}_{(C_{ij})_u}(\mathscr{Q}_{ij})_u,$$
    by Lemma \ref{lemma:blow-up-fibres}.

    Let us now show that each fibre is Fano. Let $X$ be the fibre of $\mathscr{X}_{ij}$ over any point $u\in U_{ij}$; then by the above $X\cong \mathrm{Bl}_{C}Q$ for some quadric threefold $Q$ and quartic curve $C$ obtained as a hyperplane section of the (possibly reducible) surface $(\mathscr{S}_{ij})_u\subset\P^5$. Write $H$ for the pullback of a hyperplane section of $Q$ to $X$, and $E$ the exceptional divisor of $X\to Q$. Then since the ideal of $(\mathscr{S}_{ij})_u$ in $\P^5$ is generated by 6 quadrics which are generically smooth along $(\mathscr{S}_{ij})_u$, then the ideal of $C$ in $Q$ is generated by $5$ quadrics generically smooth along $C$. Then since $X$ is the blow-up of $Q$ along $C$, it admits an embedding into $\P^4\times\P^4$ with $\O_{\P^4\times\P^4}(1,0)|_X\cong\O_X( H)$ and $\O_{\P^4\times\P^4}(0,1)|_X\cong\O_X( 2H-E)$. Then since each $Q$ is smooth along $C$, the adjunction formula for the blow-up gives:
    $$
    -K_X\sim \O_X(3H-E)\cong \O_{\P^4\times\P^4}(1,1)|_X
    $$
    which is ample.

    Finally, let us prove that the morphism $\mathscr{X}_{ij}\to U_{ij}$ has the properties described in the statement of the lemma. Clearly it is proper. For flatness and Gorensteinness, since $\mathscr{C}_{ij}\hookrightarrow\mathscr{Q}_{ij}$ is a regular embedding, the blow-up morphism $\mathscr{X}_{ij}\to \mathscr{Q}_{ij}$ is an lci morphism by \cite[Lemma 2.1]{aluffi}. The variety $\mathscr{Q}_{ij}$ is a complete intersection inside of the smooth variety $\P^5\times U_{ij}$, hence $\mathscr{X}_{ij}$ is an lci variety since the composition of lci morphisms is lci.
    
    This implies that $K_{\mathscr{X}_{ij}/U_{ij}}\cong K_{\mathscr{X}_{ij}}$ is Cartier and that $\mathscr{X}_{ij}$ is Cohen-Macaulay; thus $\mathscr{X}_{ij}\to U_{ij}$ is flat by Miracle Flatness. 
\end{proof}
We used the following lemma, which is known to experts:
\begin{lemma}
\label{lemma:blow-up-fibres}
Let $Z\hookrightarrow X$ be a regular closed immersion of Noetherian schemes, and let $X\to S$ be a morphism to a Noetherian scheme $S$. Suppose that the composition $\pi\colon Z\hookrightarrow X \to S$ is flat. Then for every $s\in S$, the fibre of $\mathrm{Bl}_ZX$ over $s$ is isomorphic to $\mathrm{Bl}_{Z_s}X_s$.
\end{lemma}
\begin{proof}
Let $s\in S$, and write $X_s$ (resp. $Z_s$) for the scheme-theoretic fibre of $X$ (resp. $Z$) over $s$. Let $ \mathcal I_s:=\mathcal I_{Z_s/X_s}$ and $\mathcal I =\mathcal I_{Z/X}$. For $\mathcal F$ a coherent sheaf of $\O_X$-modules, write $\mathcal F\otimes k(s)$ for $\mathcal F\otimes_{\pi^{-1}\O_S}\pi^{-1} k(s)$

Note that $\mathcal I\otimes k(s)\cong \mathcal I_{s}$. Indeed since $\O_Z$ is flat as a $\pi^{-1}\O_S$-module, the exact sequence 
$$0\longrightarrow\mathcal I\longrightarrow \O_X\longrightarrow\O_Z\longrightarrow 0$$
remains exact after tensoring with $k(s)$.

Now let us show that $\O_X/\mathcal I ^n$ is flat over $S$ for every $n\geq1$. For this we work with sequences of the form:

\begin{equation}
    \label{eq:ses}
0\longrightarrow\mathcal I ^n/\mathcal I ^{n+1}\longrightarrow \O_X/\mathcal I ^{n+1}\longrightarrow \O_X/\mathcal I ^{n}\longrightarrow 0.
\end{equation}
Since $Z\hookrightarrow X$ is a regular embedding, the conormal sheaf $\mathcal I/\mathcal I^2$ is a locally free $\O_Z$-module, then the sheaf $\mathrm{Sym}_{\O_Z}^n(\mathcal I/\mathcal I^2)\cong \mathcal I ^n/\mathcal I^{n+1}$ is also locally free (cf. \cite[Tag 063H]{stacks-project}), and hence flat over $Z$. Since $Z\to S$ is flat, this sheaf is therefore flat over $S$. Then since $\O_X/\mathcal I \cong\O_Z$ is flat, $\O_X/\mathcal I^2$ is flat by the exact sequence (\ref{eq:ses}) for $n=1$, and so the claim follows by induction on $n$.

We now show that every power of $\mathcal I$ commutes with restriction to the fibre. Since $\mathcal O_X/\mathcal I^n$ is flat over $S$, tensoring the exact sequence
$$
0\longrightarrow
\mathcal I^n
\longrightarrow
\mathcal O_X
\longrightarrow
\mathcal O_X/\mathcal I^n
\longrightarrow 0.
$$
by $k(s)$ gives an injective morphism
$$\mathcal I^n\otimes k(s)\longrightarrow \O_X\otimes k(s)\cong \O_{X_s}$$
for every $n\geq1$. The image of this morphism is the extension  of the ideal $\mathcal I^n$ under the map $\O_X\to\O_{X_s}$. Since extension of ideals commutes with taking powers, we have
\begin{align*}
\operatorname{im}\bigl(
\mathcal I^n\otimes k(s)\to\mathcal O_{X_s}
\bigr)
&=
\left(
\operatorname{im}\bigl(
\mathcal I\otimes k(s)\to\mathcal O_{X_s}
\bigr)
\right)^n \\ 
&=\mathcal I_s^n.
\end{align*}
By injectivity of the above morphism it follows that $\mathcal I^n\otimes k(s)
\cong
\mathcal I_s^n $ for every $n\geq 1$. Putting all this together,  and using the fact that relative $\operatorname{Proj}$ commutes with base change along $X_s\hookrightarrow X$, we obtain
\begin{align*}
(\operatorname{Bl}_Z X)_s
&\cong
\operatorname{Proj}_{X_s}
\left(
\left(
\bigoplus_{n\geq 0}\mathcal I^n
\right)\otimes k(s)
\right)\\
&\cong
\operatorname{Proj}_{X_s}
\left(
\bigoplus_{n\geq 0}\mathcal I_s^n
\right)\\
&=
\operatorname{Bl}_{Z_s}X_s.
\end{align*}
\end{proof}
We now show that these families satisfy property $(3)$ from page 3, which says that for every $p\in T$ the threefold $(\mathscr{X}_{ij})_p$ does not depend on $i$ or $j$. Namely, we show the following: 
\begin{lemma}
\label{lemma:fibrewise-compatibility}
    For every $p\in U_{i_1j_1}\cap U_{i_2j_2} $, the fibres $(\mathscr{X}_{i_1j_1})_p$ and $(\mathscr{X}_{i_2j_2})_p$ are isomorphic.
\end{lemma}
\begin{proof}
    Note that for each $i=1,\ldots,6$ we may choose affine coordinates $t_{ij},v_{ij}$ on $U_{ij}$ for $j=1,2$ which are related over $U_{i1}\cap U_{i2}$ by:
\begin{equation}
\label{eq:coordinates}
t_{i2}=t_{i1}v_{i1}, v_{i2}=v_{i1}^{-1}.
\end{equation}
Let $x\in U_{i1}\cap U_{i2}$. Since $v_{i1}(x)\neq0$, we may choose
a square root of $v_{i1}(x)$. Then the transformation
$$
\begin{aligned}
\Bigl(
[x_0:x_1:x_2:x_3:x_4:x_5],t_{i1},v_{i1}
\Bigr)
\mapsto
\Biggl(
\left[
x_0:
\frac{x_1}{\sqrt{v_{i1}}}:
x_2:
\frac{x_3}{\sqrt{v_{i1}}}:
x_4:
\frac{x_5}{v_{i1}}
\right],
v_{i1}t_{i1},
\frac{1}{v_{i1}}
\Biggr)
\end{aligned}
$$
identifies the base schemes and defining equations associated with
the two fibres. It therefore lifts to an isomorphism
$$
(\mathscr X_{i1})_x
\xrightarrow{\sim}
(\mathscr X_{i2})_x.
$$
Here this formula is understood fibrewise; the choice of square root
need not vary algebraically over the whole overlap.

We may therefore assume that $j_1=j_2$. In other words, it remains to prove that $(\mathscr{X}_{i_1j})_x\cong(\mathscr{X}_{i_2j})_x$ for every $x\in U_{i_1j}\cap U_{i_2j}$, and every $i_1,i_2\in\{1,\ldots,6\}$ and $j\in\{1,2\}$.

For this, note that each $U_{i_1j}\cap U_{i_2j}$ is contained in $U_0$ when $i_1\neq i_2$, and hence it is enough to  show that for every $x\in U_{ij}\setminus (E_i\cap U_{ij})$ the fibres of $\mathscr{X}_{ij}$ and $\mathscr{X}_0$ over $x$ are isomorphic. By $\Gamma$-equivariance, it suffices to show that this holds just for one $i$, say the open set $U_1$ corresponding to the point $p_1=[1:0:-1]$. Indeed, the group $\Gamma$ acts transitively on the set $\Sigma$, and hence its induced action on the open cover $\{U_i\}$ is transitive. Thus, for any $x\in U_{ij}\setminus (E_i\cap U_{ij})$ there exists a $g\in \Gamma$ such that $g(x)\in U_{1j}\setminus (E_1\cap U_{1j})$, yielding isomorphisms of fibres:
\begin{equation}
\label{eq:fibre isomorphism}
(\mathscr{X}_{ij})_x\cong( \mathscr{X}_{1j})_{g(x)}\cong(\mathscr{X}_0)_{g(x)}\cong(\mathscr{X}_0)_x,    
\end{equation}
for $j=1,2$. To prove the isomorphism statement over $U_0\cap U_1$, observe that we may choose affine coordinates $(t_{11},v_{11})$ and $(t_{12},v_{12})$ on $U_{11}$ and $U_{12}$ respectively such that:
\begin{enumerate}
    \item the blow-up map $\varepsilon$ is given locally on $U_{11}$ by:
$$(t_{11},v_{11})\mapsto[t_{11}+1:2t_{11}v_{11}:t_{11}-1],$$
and on $U_{12}$ by:
$$(t_{12},v_{12})\mapsto[t_{12}v_{12}+1:2t_{12}:t_{12}v_{12}-1],$$
    \item the coordinates $(t_{11},v_{11})$ and $(t_{12},v_{12})$ are related over $U_{11}\cap U_{12}$ by (\ref{eq:coordinates}). 
\end{enumerate}
Then for any $x=(t_{11},v_{11})\in U_{11}\setminus ( E_1\cap U_{11})$, there exists an isomorphism $(\mathscr{X}_{11})_x\cong \mathscr{X}_{\varepsilon(x)}\cong(\mathscr{X}_0)_x$ given by the scaling operation:
$$\big([x_0:x_1:x_2:x_3:x_4:x_5],t_{11},v_{11}\big)\mapsto \big([x_0:t_{11}x_1:t_{11}x_2:t_{11}x_3:x_4:t_{11}x_5],t_{11},v_{11}\big) $$
followed by the action (in the sense of (\ref{subsec:group action})) by the matrix
$$
A_{t_{11}}:=
\begin{pmatrix}
1 & 0 & -1 \\
0 & \sqrt{2/t_{11}} & 0 \\
1 & 0 & 1
\end{pmatrix}.
$$
A direct calculation shows that the resulting projective
transformation identifies the base scheme and defining equations of
$(\mathscr X_{11})_x$ with those of
$\mathscr X_{\varepsilon(x)}$. It therefore lifts to the corresponding
blow-ups and induces an isomorphism
$$
(\mathscr X_{11})_x
\cong
\mathscr X_{\varepsilon(x)}
\cong
(\mathscr X_0)_x.
$$
For the family $\mathscr{X}_{12}\to U_{12}$, we obtain isomorphisms $(\mathscr{X}_{12})_x\cong (\mathscr{X}_0)_x$ for every $x=(t_{12},v_{12})\in U_{12}\setminus U_{12}\cap E_1$ just as above by swapping $t_{11}$ with $t_{12}$.
\end{proof}
\begin{remark}[Geometric description of the fibres of $\mathscr{X}_{ij}\to U_{ij}$]
\label{remark:fibres-description}
    By Lemma \ref{lemma:blow-up-fibres}, the fibres of $\{\mathscr{X}_{ij}\to U_{ij}\}$ are threefolds of the form $\mathrm{Bl}_C(Q)$, where $Q\subset\P^4$ is an irreducible quadric threefold and $C\subset Q$ is a (not necessarily integral) curve of degree four along which $Q$ is smooth. Away from the exceptional curve $E_i\subset U_{ij}$, by the results of the preceding paragraphs the fibres of $\mathscr{X}_{ij}\to U_{ij}$ are isomorphic to fibres of $\mathscr{X}_0\to U_0$, which are described in Lemma \ref{lemma:Fanoness}. On the other hand, for fibres of $\mathscr{X}_{i1}$ over $E_i\cap U_{i1}$ the quadric $Q$ is given by the equations:
    $$
    \begin{cases}
        \dfrac{1}{2}(x_0+x_4) + v_{i1}x_2 - x_5 = 0, \vspace{0.5em}\\
        \dfrac{1}{2}(x_0x_2+x_2x_4) + v_{i1}x_0x_4 - x_1x_3 + x_2x_5  = 0,
    \end{cases}
    $$    
    and the curve $C$ is the union of two disjoint lines and a conic, given by:
    $$\{x_0=x_1=x_2=x_4-2x_5=0\}\cup \{x_0=x_4= x_5-v_{i1}x_2=x_1x_3-v_{i1}x_2^2=0\}\cup \{x_2=x_3=x_4= x_0-2x_5=0\}.$$
    For the fibres of $\mathscr{X}_{i2}$ over $E_i\cap U_{i2}$ the quadric $Q$ is given by the equations:
    $$
    \begin{cases}
        \dfrac{v_{i2}}{2}(x_0+x_4) + x_2 - x_5 = 0, \vspace{0.5em}\\
        \dfrac{v_{i2}}{2}(x_0x_2+x_2x_4) + x_0x_4 - x_1x_3 + x_2x_5  = 0,
    \end{cases}
    $$    
    and the curve $C$ is given by:
    $$
    \{x_0=x_1=x_2=2x_5-v_{i2}x_4=0\}\cup \{x_0=x_4= x_5-x_2= x_1x_3-x_2^2=0\}\cup\{x_2=x_3=x_4= 2x_5-v_{i2}x_0=0\}.
    $$
    Then for each value of $v_{i1}\in E_i\cong\P^1$, we have the following:
\begin{enumerate}
    \item $\mathrm{Bl}_{C_{1,2,1}}(Q^\mathrm{sm})$, the general fibre,
    \item $\mathrm{Bl}_{C_{1,2,1}}(Q^\mathrm{sm})$, the fibre over $v_{i1}=v_{i2}^{-1}=\infty$, which is invariant under the action of the torus in $\mathrm{PGL}_3(\C)$ given by:
    $$
    \begin{pmatrix}
        1 & 0 & 0 \\ 
        0 & \lambda &  0 \\
        0 & 0 &\lambda^2
    \end{pmatrix},
    $$
    \item $\mathrm{Bl}_{C_{1,2,1}}(Q^\mathrm{iso})$, the fibre over $v_{i1}=\pm 1$,
    \item $\mathrm{Bl}_{C_{1,1,1,1}}(Q^\mathrm{iso})$, the fibre over $v_{i1}=0$.
\end{enumerate}
\end{remark}
To complete the proof that the families $\{\mathscr{X}_{ij}\to U_{ij}\}$ satisfy the desired properties, it remains to show that the fibre-isomorphism relation between the members is given precisely by orbits of the $\Gamma$-action on $T$.
\begin{lemma}
\label{lemma:T-orbit-relation}
    For all $x,y\in T$, and any $U_{i_1j_1}\ni x$ and $U_{i_2,j_2}\ni y$, we have that $y\in \Gamma\cdot x$ if and only if $(\mathscr{X}_{i_1j_1})_x\cong (\mathscr{X}_{i_2j_2})_y$.
\end{lemma}
\begin{proof}
    Since $U_0\subset T$ is $\Gamma$-stable and fibres of $\{\mathscr{X}_{ij}\to U_{ij}\}$ over points of $\Delta_2=T\setminus U_0$ are never isomorphic to fibres of $\mathscr{X}_0\to U_0$, then we may assume that either $x,y\in U_0$ or $x,y \in \Delta_2$. Assume first the former. Then by (\ref{eq:fibre isomorphism}), there exist isomorphisms:
    \begin{align*}
        (\mathscr{X}_{i_1j_1})_x&\cong \mathscr{X}_{\e(x)}, \\
        (\mathscr{X}_{i_2j_2})_y&\cong \mathscr{X}_{\e(y)}
    \end{align*}
    for any open subsets $U_{i_1j_1}\ni x$ and $U_{i_2j_2}\ni y$ in the cover. By Lemma \ref{lemma:orbits} we have that $(\mathscr{X}_{i_1j_1})_x\cong(\mathscr{X}_{i_2j_2})_y$ if and only if $\e(y)\in \Gamma\cdot \e(x)$, and hence if and only if $y\in \Gamma\cdot x$.

    Now suppose that $x,y\in \Delta_2$; therefore $x\in E_1$ and $y\in E_i$ for some $i$, after renumbering the indices of the exceptional divisors $\{E_i\,|\,i=1,\ldots,6\}.$ The stabiliser $\Gamma_{p_1}$ of the point $p_1$ is generated by the matrices $M_1,M_2$, which act via the induced projective representation on $E_1\cong \P(T_{p_1}\P^2)$ by:
    \begin{align*}
        M_1\colon v_{11}, v_{12}&\mapsto-v_{11},-v_{12} \\
        M_2\colon v_{11}, v_{12}&\mapsto v_{11}, v_{12}. 
    \end{align*}
    There is also an action of the group $\Gamma_{p_1}$ on the total spaces of the restricted families over $E_1$:
    $$\mathscr{X}_{11}\times_{U_{11}}(E_1\cap U_{11})\to E_1\cap U_{11}\quad \text{ and } \quad\mathscr{X}_{12}\times_{U_{12}}(E_1\cap U_{12})\to E_1\cap U_{12}$$
    such that they are $\Gamma_{p_1}$-equivariant, given by the projective transformations:
    \begin{align*}
    M_1\colon \mathscr{X}_{11}\times_{U_{11}}(E_1\cap U_{11}) &\longrightarrow \mathscr{X}_{11}\times_{U_{11}}(E_1\cap U_{11}), \\
    \big([x_0:x_1:x_2:x_3:x_4:x_5],  v_{11}\big) &\longmapsto \big([x_0:ix_1:-x_2:ix_3:x_4:x_5],  -v_{11}\big),
    \end{align*}
    and 
    \begin{align*}
    M_1\colon \mathscr{X}_{12}\times_{U_{12}}(E_1\cap U_{12}) &\longrightarrow \mathscr{X}_{12}\times_{U_{12}}(E_1\cap U_{12}), \\
    \big([x_0:x_1:x_2:x_3:x_4:x_5],  v_{12}\big) &\longmapsto \big([-x_0:x_1:x_2:x_3:-x_4:x_5],  -v_{12}\big)        
    \end{align*}
    while $M_2$ acts trivially on the base $E_1$ and fibrewise by the projective transformation:
    $$\big([x_0:x_1:x_2:x_3:x_4:x_5], v_{1j}\big)\longmapsto\big( [x_4:x_3:x_2:x_1:x_0:x_5], v_{1j}\big)$$
    for $j=1,2$. We thus have that fibres of $\mathscr{X}_{11}\times_{U_{11}}(E_1\cap U_{11})$ and $\mathscr{X}_{12}\times_{U_{12}}(E_1\cap U_{12})$ over each $\Gamma_{p_1}$-orbit in $E_1$ are isomorphic, and so fibres of the restricted families
    $$\big\{\mathscr{X}_{ij}\times_{U_{ij}}(E_i\cap U_{ij})\to E_i\cap U_{ij}\big\}$$ over each $\Gamma$-orbit in $\Delta_2$ are isomorphic. 
    Now suppose that there exists an abstract isomorphism
$$
\phi\colon (\mathscr{X}_{1j_1})_x\xrightarrow{\sim}(\mathscr{X}_{ij_2})_y.
$$
Choose $g\in\Gamma$ such that $g(E_i)=E_1$. By the $\Gamma$-equivariance of the construction, $g$ induces an isomorphism from $(\mathscr{X}_{ij_2})_y$ to the corresponding fibre over $g(y)\in E_1$. Replacing $y$ by $g(y)$ and composing $\phi$ with this isomorphism, we may therefore assume that $x,y\in E_1$.
Note that the point $0\in E_1$ given by $v_{11}=0$ is a fixed-point for the $\Gamma_{p_1}$-action. Moreover, the fibre $(\mathscr{X}_{11})_0$ over it has four isolated ordinary double points, whereas the fibres of $\mathscr{X}_{11}$ and $\mathscr{X}_{12}$ over every other point of $E_1$ have just two or three singular points, and hence are not isomorphic to $(\mathscr{X}_{11})_0$. Thus, we may assume that $x,y\in E_1\cap U_{12}$. After possibly replacing $\phi$ by $
\phi\circ\widetilde{\psi}_{12}|_{(\mathscr{X}_{12})_x}, $
where $\widetilde{\psi}_{12}$ exchanges the two extremal rays of $\mathrm{Nef}((\mathscr{X}_{12})_x)$, we may assume that $\phi$ preserves the extremal ray corresponding to the blow-down morphisms
$$
(\mathscr{X}_{12})_x
\cong
\mathrm{Bl}_{(\mathscr{C}_{12})_x}(\mathscr{Q}_{12})_x
\longrightarrow
(\mathscr{Q}_{12})_x,\quad\text{ and }\quad
(\mathscr{X}_{12})_y
\cong
\mathrm{Bl}_{(\mathscr{C}_{12})_y}(\mathscr{Q}_{12})_y
\longrightarrow
(\mathscr{Q}_{12})_y.
$$
It follows that $\phi$ descends to an isomorphism
$$
\overline{\phi}\colon
(\mathscr{Q}_{12})_x\xrightarrow{\sim}(\mathscr{Q}_{12})_y
$$
which maps $(\mathscr{C}_{12})_x$ isomorphically onto $(\mathscr{C}_{12})_y$. Moreover, $\overline{\phi}$ preserves the hyperplane class, and hence is induced by a projective linear isomorphism of the ambient $\P^4$'s, which may be extended to a projective transformation of $\P^5$. Using the explicit equations for $(\mathscr{Q}_{12})_x$ and $(\mathscr{C}_{12})_x$ given in Remark \ref{remark:fibres-description}, a direct calculation shows that
$$
v_{12}(y)=\pm v_{12}(x).
$$
Since the nontrivial element in the effective action of $\Gamma_{p_1}$ on $E_1$ sends $v_{12}\mapsto-v_{12}$, it follows that $x$ and $y$ lie in the same $\Gamma_{p_1}$-orbit.
\end{proof}

\section{K-polystability of the fibres}
\label{section:K-polystability}
In this section, we show that every fibre of the families $\{\mathscr{X}_{ij}\to U_{ij}\}$ is K-polystable. Let us first consider their smooth fibres.
\begin{lemma}
\label{lemma:smooth}
    Let $X$ be a fibre of $\{\mathscr{X}_{ij}\to U_{ij}\}$, and suppose that $X$ is smooth. Then $X$ is K-polystable.    
\end{lemma}
\begin{proof}
    By (\ref{eq:fibre isomorphism}), we have that $X$ is isomorphic to a fibre of the family $\mathscr{X}_0\to U_0$. Then the smooth quadric threefold $Q$, which $X$ is the blow-up of, is invariant under the action of the subgroup of $\mathrm{PGL}_3(\C)$ generated by:
$$\begin{pmatrix}
    1 & 0 & 0 \\ 
    0 & 1  & 0 \\
    0 & 0 & -1
\end{pmatrix},\quad
\begin{pmatrix}
    1 & 0 & 0 \\ 
    0 & -1 & 0 \\
    0 & 0 & 1
\end{pmatrix}
$$
which is isomorphic to $\Z_2\times\Z_2$. Therefore $\Z_2\times\Z_2\subseteq\mathrm{Aut}(Q,C_4)$, and so $X$ is K-polystable by \cite[Main Theorem]{K-STAB}.
\end{proof}

For the remainder of this section, we prove that all singular fibres of $\{\mathscr{X}_{ij}\to U_{ij}\}$ are K-polystable. We will first prove, in Lemma \ref{lemma:isolated}, that fibres with isolated singularities are K-polystable. Then in Lemma \ref{lemma:non-isolated} we prove that the unique (up to isomorphism) fibre with non-isolated singularities is K-polystable.

\subsection{Preliminaries}
Let us prepare some preliminary results which can be used in most cases. If $X$ is a threefold appearing as a fibre of $\{\mathscr{X}_{ij}\to U_{ij}\}$, with involution $\chi$, then we write $\mathrm{Ex}(X)$ for the union of all the exceptional divisors of $X$; that is the divisor,
$$\mathrm{Ex}(X):=E \cup \chi(E),$$
where $E$ is the exceptional locus of the blow-up $X\to Q$.

\begin{lemma}
    \label{lemma:delpezzo}
    Let $Q\subset \P^4$ be a quadric which is either smooth or has an isolated singularity, and let $X$ be the blow-up along a curve $C\subset Q$ which is one of the reduced curves listed in Theorem \ref{theorem:main}. Fix a smooth point $x\in X\setminus\mathrm{Ex}(X)$, and let $S$ be the pullback of a general hyperplane section of $Q$ containing $x$. Then:
    \begin{enumerate}
        \item[$(\mathrm{i})$] $S$ is a smooth del Pezzo surface of degree 4,
        \item[$(\mathrm{ii})$] The point $x$ doesn't lie on any exceptional curve in $S$.
    \end{enumerate}
\end{lemma}

\begin{proof}
    First we prove $(\mathrm{i})$. Since $Q$ has isolated singularities,  the image $\pi(S)$ of $S$ on $Q$ is smooth by Bertini's theorem. Since $C$ is reduced and $S$ is general, then $\pi(S)$ intersects $C$ transversally in four points; hence $S$ is smooth and moreover $(-K_S)^2=4$. We show that $-K_S$ is ample. Since $S$ is the blow-up of a hyperplane section $H\cong\P^1\times\P^1$ along $\Sigma:=H\cap C$, we must show that:
    \begin{enumerate}
        \item No two points of $\Sigma$ lie on a line in $H$.  
        \item $\Sigma$ is not contained in a plane.
    \end{enumerate}
    For $(1)$, observe that hyperplane sections of $Q$ containing $\pi(x)$ are naturally parametrised by $\mathcal{P}\cong\P^3$. We must show that for a general $H\in \mathcal{P}$, $H$ does not contain any secant of $C$ (we understand ``secant line" to include the case of tangent lines). Under the involution $\tilde{\chi}$ described in Section \ref{section:families} (or $\tilde{\chi}_0$ in the case where $C=C_{1,2,1}$ or $C_{1,1,1,1}$), the secants and line-components of  of curve $C$ are naturally parametrised by a dense open subset $U\subset C$. Let $\mathcal{I}\subset \mathcal{P}\times U$ be the incidence subvariety: 
    $$\mathcal{I}=\bigg\{\big(H, L)\,\bigg|\,H\in \mathcal{P}, L\subset H,  L \text{ a secant of }C\bigg\}.$$
    Let $\mathrm{pr}_1 \colon\mathcal{I}\to \mathcal{P}, \mathrm{pr}_2 \colon \mathcal{I}\to U$ be the natural projections. The subvariety $\mathrm{pr}_1(\mathcal{I})$ of $\mathcal{P}$ parametrises hyperplane sections of $Q$ containing a secant of $C$. For a point $[L]\in U$ the fibre $\mathrm{pr}_2^{-1}([L])$ parametrises hyperplanes containing $\pi(x)$ and a given secant, and hence $\mathrm{pr}_2^{-1}([L])\cong \P^1$ since $\pi(x)$ is not contained in a secant of $C$. Therefore, the fibres of $\mathrm{pr}_2$ are all isomorphic to $\P^1$, hence $\mathrm{dim}(\mathcal{I})=2$ since $\mathrm{dim}(U)=1$, and so $(1)$ follows.
    
    For $(2)$, let $\Pi\subset\P^4$ be the ambient hyperplane whose restriction to $Q$ is $H$, so that $\Pi\cap C=\Sigma$. Since $\Pi$ contains no component of $C$, its defining equation is a non-zero-divisor on $\O_C$. The hyperplane restriction exact sequence therefore gives
    $$0\longrightarrow\mathcal{J}_C\longrightarrow\mathcal{J}_C(1)\longrightarrow\mathcal{J}_{\Sigma/\Pi}(1)\longrightarrow 0.$$
    Here $\mathcal{J}_C$ denotes the ideal sheaf of $C$ in $\P^4$, and $\mathcal{J}_{\Sigma/\Pi}$ denotes the ideal sheaf of $\Sigma$ in $\Pi$ pushed forward to $\P^4.$ The group $H^0(C,\O_C)=\C$ since $C$ is connected, reduced and projective, hence $H^1(\P^4, \mathcal{J}_C)$ vanishes. Moreover, the group $H^0(\P^4, \mathcal{J}_C(1))$ vanishes since $C$ is non-degenerate, and so $H^0(\Pi, \mathcal{J}_{\Sigma/\Pi}(1))=0$. Therefore, $\Sigma$ does not span a plane in $\Pi$. \\ 

    Now we prove $(\mathrm{ii})$. We must show that:
    \begin{enumerate}
        \item The four lines joining $\pi(x)$ to each point of $\Sigma$ do not lie in $Q$;
        \item $\pi(x)$ is not contained in the plane spanned by any three points of $\Sigma$. \\ 
    \end{enumerate}
   
    For $(1)$, note that the union of lines in $Q$ meeting ${\pi(x)}$ is given by $Q\cap T_{\pi(x)} Q$. This surface meets $C$ in four points, so there exist just four lines in $Q$ through $\pi(x)$ which meet $C$. A general hyperplane section $H$ of $Q$ containing $\pi(x)$ does not contain any of these lines, so $(1)$ is proven. 
    
    For $(2)$, let $\P^4\dashrightarrow\P^3$ be the projection from ${\pi(x)}$ and observe that there is a natural isomorphism between $\mathcal{P}$ and $|\O_{\P^3}(1)|$. Moreover since ${\pi(x)}$ lies on no secant of $C$, $C$ maps to an isomorphic curve in $\P^3$. So fix $H\in\mathcal{P}$, and suppose $H$ contains a plane through $\pi(x)$ and three points of $C$. Projecting to $\P^3$, the image of $H$ is a plane in $\P^3$ containing a trisecant $L$ of $C\subset\P^3$. We claim that if $H$ is general then this cannot happen.

    For this, note that trisecants of $C$ are parametrised by a one-dimensional variety. Indeed, the linear system $|\mathcal{I}_{C/\P^3}(3)|$ has base locus equal to $C$, and defines a birational map $\P^3\dashrightarrow Y\subset\P^6$ which contracts trisecants of $C$ onto a curve $Z\subset Y$. Now let us define the incidence subvariety:

    $$\mathcal{I}=\bigg\{\big(H, L)\,\bigg|\,H\in |\O_{\P^3}(1)|, L\subset H,  L \text{ a trisecant of $C_4$}\bigg\}.$$
    Let $\mathrm{pr}_1 \colon\mathcal{I}\to |\O_{\P^3}(1)|, \mathrm{pr}_2 \colon \mathcal{I}\to Z$ be the natural projections. Then the fibres of $\mathrm{pr_2}$ are 1-dimensional, so that $\mathrm{dim}(\mathcal{I})=2$, and hence $\mathrm{dim}(\mathrm{pr}_1(\mathcal{I}))\leq2$. Thus, we see that a general plane in $\P^3$ does not contain a trisecant of $C$. 
    \end{proof}
    
\begin{cor}
    \label{cor:delta}
    With the notation of Lemma \ref{lemma:delpezzo}, we have that $\delta_x(X)>1$. 
\end{cor}

\begin{proof}
    Let $S$ be the pullback of a general hyperplane section of $Q$ containing $x$. Then $S$ is a smooth del Pezzo surface of degree 4 and $x$ does not lie on any $(-1)$-curve on $S$ by Lemma \ref{lemma:delpezzo}, so that by following the proof of \cite[Lemma 14]{K-STAB} verbatim we see that $\delta_x(X)=\frac{112}{111}>1.$
\end{proof}

\begin{lemma}
    \label{lemma:1ODP}
    Suppose that $Q\subset \P^4$ is an irreducible quadric hypersurface, and let $\pi\colon X\to Q$ be the blow-up along a curve $C$ which is either a rational normal quartic curve or a curve of type $C_{1,2,1}$ as described in Theorem \ref{theorem:main}. Suppose that $\mathrm{Sing}(X)$ consists of a single ordinary double point $x$ not contained in $\mathrm{Ex}(X)$. Let $f\colon\wt{X}\to X$ be the blow-up at $x$, $E$ be the exceptional divisor, and $\wt{Z}\subset E\cong\P^1\times\P^1$ an irreducible curve which is not a fibre of either projection to $\P^1$. Then $\delta_{\wt{Z}}(f^*(-K_X))>1$.   
\end{lemma}

\begin{proof}
     By \cite[Theorem 3.3]{AZ22} we have that $$\delta_{\wt{Z}}(f^*(-K_X))\geq\mathrm{min}\bigg\{\frac{A_X(E)}{S(f^*(-K_X);E)},\frac{1}{S(W_{\bullet,\bullet};{\wt{Z}})}\bigg\}.$$
     We have that $A_X(E)=2$. Denote by $V\subset Q$ the surface obtained by taking the closure of the union of lines in $Q$ which connect $\pi(x)$ to a point of $C$, and let $\wt{V}\subset \wt{X}$ be its strict transform. Since $x\notin\mathrm{Ex}(X)$, the point $\pi(x)$ lies on no secant line of
$C$ (we understand ``secant line" to include the case of tangent lines). Therefore, projection from $\pi(x)$ induces an isomorphism
of $C$ onto a curve in
$
E\simeq\P^1\times\P^1
$
of bidegree $(1,3)$ or $(3,1)$ (without loss of generality, we assume the former case). It follows that $\wt{V}\sim H_1+3H_2-F$, where $F$ is the exceptional divisor over $C$ and $H_1,H_2$ are strict transforms of the two generators of $\mathrm{Cl}(Q)$. We have the following Zariski decomposition $f^*(-K_X)-uE=P(u)+N(u)$ where 
     \begin{align*}
         P(u)&\sim_{\mathbb{R}}\begin{cases}
         3(H_1+H_2)+(3-u)E-F, & u\in[0,2],  \\ 
         (5-u)H_1+(3-u)(3H_2+E-F),&u\in[2,3],
         \end{cases} \\ 
         N(u)&=\begin{cases}
         0, &u\in[0,2], \\ 
         (u-2)\wt{V}, &u\in[2,3],
         \end{cases} 
     \end{align*}
     which has the property that

     $$\mathrm{vol}\big(f^*(-K_X)-uE\big)=\mathrm{vol}\big(P(u)\big)=P(u)^3$$
    for every $u\in[0,3]$. Before continuing with the computation, let us first justify this decomposition. 
    The divisor $P(2)\sim f^*(-K_X)-2E$ is nef because its complete linear system corresponds to double projection from $x\in X$ under the anticanonical embedding. Thus $P(u)$ is nef for $u\in [0,2]$, since it can be written as a convex sum of the nef divisors $P(0)$ and $P(2)$. For $u\in [2,3]$, we may write:
     $$P(u)=(3-u)(f^*(-K_X)-2E)+2(u-2)H_1,$$
     which shows that $P(u)$ is nef since $H_1$ is. Moreover, the birational morphism given by the linear system $|P(2)|$ contracts the strict transforms of (anticanonical) conics in $X$ containing $x$ which intersect $E$ transversely in one point. By the definition of $V$, these curves span the surface $\wt{V}$, which implies that $\wt{V}$ is exceptional for the map corresponding to $|P(2)|$. Moreover, we have that $H_1$ is trivial on the strict transforms of these conics, and hence is pulled back from a Cartier divisor on the base. Thus, we have that for $u\in[2,3]$:
     \begin{align*}
        \mathrm{vol}\big(f^*(-K_X)-uE\big)&=\mathrm{vol}\big((3-u)P(2)+2(u-2)H_1+(u-2)\wt{V}\big) \\ 
        &=\mathrm{vol}\big((3-u)P(2)+2(u-2)H_1\big) \\ 
        &=\big((3-u)P(2)+2(u-2)H_1\big)^3
     \end{align*}
     We have the following intersection numbers on $\wt{X}$:
     $$H_1^3=0,H_2^3=0,H_1^2H_2=0,H_1H_2^2=0,F^3=-10,EF^2=0,E^2F=0,E^3=2,$$
     $$E^2H_1=-1,E^2H_2=-1,EFH_1=0,EFH_2=0,EH_1^2=0,EH_2^2=0,EH_1H_2=1,$$
     $$F^2H_1=-3,F^2H_2=-1,FH_1^2=0,FH_2^2=0,FH_1H_2=0.$$
     Then:
     \begin{align*}
         S(f^*(-K_X); E)&=\frac{1}{\mathrm{vol}(-K_X)}\int_0^\infty\mathrm{vol}(f^*(-K_X)-uE)du \\ 
         &=\frac{1}{28}\bigg(2\int_0^2(14-u^3)du+12\int_2^3(u-3)^2du\bigg) \\ 
         &=\frac{13}{7}.
     \end{align*}
     Note that  $f^*(-K_X)-uE$ is big for $0\leq u<3$ but $\mathrm{vol}(f^*(-K_X)-3E)=0$, which implies that the pseudo-effective threshold of $f^*(-K_X)$ with respect to $E$ is equal to 3.
     
     Now we compute $S(W_{\bullet,\bullet};\wt{Z})$. Denote by $\mathbf{h}_1$ and $\mathbf{h}_2$ the restrictions of $H_1$ and $H_2$ to $E$ (i.e. the two rulings of $E\cong\P^1\times \P^1$). Then we have the following equivalences on $E$:
     \begin{align*}
         E|_E&\sim -\mathbf{h}_1-\mathbf{h}_2, \\ 
         P(u)|_E&\sim_\mathbb{R}\begin{cases}
         u\mathbf{h}_1+u\mathbf{h}_2, & u\in[0,2], \\ 
         2\mathbf{h}_1+(6-2u)\mathbf{h}_2,& u\in[2,3], \\ 
     \end{cases}    \\
     N(u)|_E &= \begin{cases}
     0, & u\in[0,2], \\ 
     (u-2)\wt{V}|_E, & u\in[2,3].
     \end{cases}
     \end{align*}
     Now, since ${\wt{Z}}$ is assumed to be not a fibre of either projection to $\P^1$, we have that the difference ${\wt{Z}}-\mathbf{h}_1-\mathbf{h}_2$ is effective, and hence that $$\mathrm{vol}\big(P(u)|_E-v{\wt{Z}}\big)\leq\mathrm{vol}\big(P(u)|_E-v\mathbf{h}_1-v\mathbf{h}_2\big).$$ 
     Similarly, since $\wt{V}|_E$ is a reduced curve linearly equivalent to $\mathbf{h}_1+3\mathbf{h}_2$ we have that $\mathrm{ord}_{\wt{Z}}\big(\wt{V}|_E\big)\leq1$. Then we have that:
     \begin{align*}
    S(W_{\bullet,\bullet};{\wt{Z}}) &= \frac{3}{28}\int_0^3(P(u)|_E)^2\mathrm{ord}_{\wt{Z}}\big(N(u)|_E\big)du +\frac{3}{28}\int_0^3\int_0^\infty\mathrm{vol}\big(P(u)|_E-v{\wt{Z}}\big)dvdu \\      
      &\leq\frac{3}{28}\Bigg(4\int_2^3(6-2u)(u-2)du 
         +2\int_0^2\int_0^u(u-v)^2 dvdu \\
         &+ 2\int_2^3\int_0^{6-2u}(2-v)(6-2u-v) dvdu\Bigg) \\ 
     &=\frac{9}{14}.
     \end{align*}
Thus 
    \begin{align*}
        \delta_{\wt{Z}}(f^*(-K_X))&\geq\mathrm{min}\bigg\{\dfrac{A_X(E)}{S(f^*(-K_X);E)},\dfrac{1}{S(W_{\bullet,\bullet};{\wt{Z}})}\bigg\} \\ 
        &\geq \mathrm{min}\bigg\{\dfrac{14}{13},\dfrac{14}{9}\bigg\} \\ 
        &>1.
    \end{align*}
\end{proof}

\subsection{Proofs}

With these results prepared, we are ready to prove K-polystability of the singular fibres.

\begin{lemma}
\label{lemma:isolated}
    Let $X$ be a singular fibre of one of the families $\{\mathscr{X}_{ij}\to U_{ij}\}$, and suppose that $X$ has isolated singularities. Then $X$ is K-polystable.
\end{lemma}
\begin{proof}
    Observe that the threefold $X$ is of one of the following five types:
    $$\mathrm{Bl}_{C_4}(Q^\mathrm{iso}),\quad\mathrm{Bl}_{C_{2,2}}(Q^\mathrm{iso}), \quad \mathrm{Bl}_{C_{1,2,1}}(Q^\mathrm{sm}), \quad \mathrm{Bl}_{C_{1,2,1}}(Q^\mathrm{iso}), \quad \mathrm{Bl}_{C_{1,1,1,1}}(Q^\mathrm{iso}).$$
    In each case, as described previously, there is a subgroup $G\subset \mathrm{Aut}(X)$ isomorphic to
    $$\big(\Z_2\times\Z_2\big)\rtimes\langle\tilde{\chi}\rangle.$$
    We will show that for every $G$-invariant prime divisor $D$ over $X$ we have that
    $$A_X(D)>S(-K_X;D),$$
    which implies that $X$ is K-polystable by \cite{Fujita2019,ZZ}. So fix a $G$-invariant prime divisor $D$ over $X$, and let $Z$ be its centre on $X$. 
    
    Suppose first that $X\cong \mathrm{Bl}_{C_4}(Q^\mathrm{iso})$. After a linear change of coordinates we may assume that the quadric $Q^\mathrm{iso}$ is given by an equation of the form:
    $$\mu(x_2x_4-x_3^2+x_0x_2-x_1^2)+\lambda(x_0x_4-x_2^2)-x_1x_3+x_2^2=0,$$
    the curve $C_4$ is given by the parametrisation $t\mapsto [1:t:t^2:t^3:t^4]$, and the group $G\cap \mathrm{Aut}(Q^\mathrm{iso},C_4)\cong \Z_2\times\Z_2$ acts on the quadric by the transformations:
    \begin{align*}
        [x_0:x_1:x_2:x_3:x_4]&\mapsto[x_4:x_3:x_2:x_1:x_0], \\ 
        [x_0:x_1:x_2:x_3:x_4]&\mapsto[x_0:-x_1:x_2:-x_3:x_4].
    \end{align*}
    In this case, let $f\colon \wt{X}\to X$ be the blow-up at $\mathrm{Sing}(X)=\{x\}$ with exceptional divisor $E$. One checks that the group $G$ acts on $E$ without fixed-points, in such a way that the two rulings $E\to \P^1$ are preserved. It follows that any $G$-invariant irreducible subvariety of $E$ is a curve which is not a fibre of either ruling (since $G$ acts without fixed points on the bases of these rulings) or is $E$ itself. 
    
    Suppose first that $Z=x$; let $\wt{Z}\subseteq E$ be its centre on $\wt{X}$. Then:
    $$\frac{A_X(D)}{S(-K_X;D)}\geq \delta_Z(-K_X)=\delta_{\wt{Z}}(f^*(-K_X))>1$$
    by Lemma \ref{lemma:1ODP}. Suppose instead that $Z$ lies inside $\mathrm{Ex}(X)$. Then following the proof of \cite[Lemma 15]{K-STAB} verbatim gives the desired inequality $A_X(D)>S(-K_X;D)$. 

    We may therefore assume that there exists a point $x\in Z\setminus(Z\cap\mathrm{Ex}(X))$ such that $X$ is smooth at $x$. Then:
    $$\frac{A_X(D)}{S(-K_X;D)}\geq\delta_x(-K_X)>1$$
    by Corollary \ref{cor:delta}. It follows that $X$ is K-polystable.\\ 

    Now suppose that $X$ is one of the remaining four isomorphism types. Then one checks that in each case there are no $G$-invariant irreducible subvarieties contained in $\mathrm{Ex}(X)$. Thus there exists a point $x\in Z\setminus (Z\cap\mathrm{Ex}(X))$. If $x$ is a singular point of $X$, then since $x\notin\mathrm{Ex}(X)$ we must have that $X\cong\mathrm{Bl}_{C_{1,2,1}}(Q^\mathrm{iso}) $ and $\pi(x)$ is the unique singular point of $Q^\mathrm{iso}$; then after a linear change of coordinates we may assume that the quadric $Q^\mathrm{iso}$ is given by the equation 
    $$x_0x_2+x_2x_4 + x_0x_4 -x_1x_3 +x_2^2 = 0,$$
    and $G\cap \mathrm{Aut}(Q^\mathrm{iso}, C_{1,2,1})$ acts on $Q^\mathrm{iso}$ as previously described. Then by arguing as in the case $\mathrm{Bl}_{C_4}(Q^\mathrm{iso})$, we obtain
    $$A_X(D)>S(-K_X;D).$$
    Thus, we may assume that there exists a point $x\in Z\setminus (Z\cap\mathrm{Ex}(X))$ which is a smooth point of $X$, so that:
    $$\frac{A_X(D)}{S(-K_X;D)}\geq \delta_x(-K_X)>1$$
    by Corollary \ref{cor:delta}. Thus
$$
A_X(D)>S(-K_X;D)
$$
for every $G$-invariant prime divisor $D$ over $X$. The equivariant
valuative criterion therefore implies that $X$ is K-polystable.
\end{proof}

It remains to prove K-polystability of the fibre with non-isolated singularities:
\begin{lemma}
\label{lemma:non-isolated}
    Let $X$ be a singular fibre of one of the families $\{\mathscr{X}_{ij}\to U_{ij}\}$, and suppose that $X$ has non-isolated singularities. Then $X$ is K-polystable.
\end{lemma}
Before proving this, let us first describe some aspects of the geometry of this threefold $X$. By construction $X$ is isomorphic to a variety of the form $\mathrm{Bl}_{C_{2,2}^\mathrm{nred}}(Q^\mathrm{niso})$. Up to a change of coordinates, the quadric $Q^\mathrm{niso}$ may be given in $\P^4$ by the equation
$$x_0x_4-x_2^2 = 0 ,$$
and the curve $C_{2,2}^\mathrm{nred}$ is a non-reduced curve with the double structure described in Theorem \ref{theorem:main} supported on the conic $C_{2,2}^\mathrm{red}$ given by:
$$x_1=x_0x_4-x_2^2=x_3=0.$$
The group $\mathrm{GL}_2(\C)$ acts on $\P^4$ via the five-dimensional representation:
\[
\begin{pmatrix}
a & b\\
c & d
\end{pmatrix}
\mapsto
\begin{pmatrix}
a^2 & 0 & 2ab & 0 & b^2\\
0 & a & 0 & b & 0\\
ac & 0 & ad+bc & 0 & bd\\
0 & c & 0 & d & 0\\
c^2 & 0 & 2cd & 0 & d^2
\end{pmatrix}.
\]
Under this $\mathrm{GL}_2(\C)$-action, the subschemes $C_{2,2}^\mathrm{red},C_{2,2}^\mathrm{nred}$ and $Q^\mathrm{niso}$ are preserved. Therefore the $\mathrm{GL}_2(\C)$-action lifts to a faithful action on $X$. Furthermore the cubic ruled surface $S$ introduced in Theorem \ref{theorem:main}, given in this coordinate system by:
$$x_0x_3-x_1x_2=x_1x_4-x_2x_3=x_0x_4-x_2^2 =0,$$
is $\mathrm{GL}_2(\C)$-invariant. Let $L=\mathrm{Sing}(Q^\mathrm{niso})$, which is the line given by $x_0=x_2=x_4=0$. It is also $\mathrm{GL}_2(\C)$-invariant. In fact, we have the following:
\begin{lemma}
\label{lemma:Q invariants}
    For the $\mathrm{GL}_2(\C)$-action on $Q^\mathrm{niso}$ described above, the only invariant proper irreducible subvarieties are:
    $$S, \quad C:=C_{2,2}^\mathrm{red},\quad  L.$$
\end{lemma}
\begin{proof}
    We prove that $Q^\mathrm{niso}$ admits the finite $\mathrm{GL}_2(\C)$-orbit decomposition:
    $$Q^\mathrm{niso}=\big(Q^\mathrm{niso}\setminus S\big)\sqcup \big(S\setminus(C\cup L)\big)\sqcup C\sqcup L.$$
    Then the lemma follows from the fact that invariant subvarieties of $Q^\mathrm{niso}$ are unions of these orbits.

    By transitivity of the $\mathrm{GL}_2(\C)$-action on $L$ and $C$, this proves that $L$ and $C$ are both orbits. Now fix $p\in Q^\mathrm{niso}$, and write $p=[s^2:x_1:st:x_3:t^2]$ for some $s,t,x_1,x_3\in\mathbb C$. If $s=t=0$, then $p\in L$, while if $x_1=x_3=0$, then $p\in C$. We may therefore assume that
$$
(s,t)\neq(0,0)
\quad\text{and}\quad
(x_1,x_3)\neq(0,0).
$$
Substituting the point $p$ into the first two equations defining $S$, we get
$$
s^2x_3-stx_1=s(sx_3-tx_1)
$$
and
$$
x_1t^2-stx_3=-t(sx_3-tx_1),
$$
respectively. Then since $(s,t)\neq(0,0)$, it follows that
$$
p\in S
\quad\Longleftrightarrow\quad
sx_3-tx_1=0.
$$
Suppose first that $p\in S\setminus(C\cup L)$; then there exists $\lambda\in\C^*$ such that $x_1=\lambda s $ and $x_3=\lambda t.$ Choose
$$g=\begin{pmatrix}a & b \\ c & d\end{pmatrix}\in\mathrm{GL}_2(\C)$$
such that $as+bt=\lambda$ and $cs+dt =0.$ Then under the above $\mathrm{GL}_2(\C)$-action, we have that:
$$g\cdot p =[1:1:0:0:0].$$
Thus every point of $S\setminus(C\cup L)$ belongs to the orbit of $[1:1:0:0:0],$ which implies that $S\setminus(C\cup L)$ is a single orbit.

Suppose now that $p\in Q^\mathrm{niso}\setminus S$. Then $sx_3-tx_1\neq0,$
and hence the matrix
$$
\begin{pmatrix}
s&x_1\\
t&x_3\\
\end{pmatrix}
$$
is invertible. Therefore by choosing
$$
g=
\begin{pmatrix}
s&x_1\\
t&x_3\\
\end{pmatrix}^{-1}
\in\mathrm{GL}_2(\mathbb C)
$$
it follows that
$$
g\cdot p=[1:0:0:1:0].
$$
Thus every point of $Q^\mathrm{niso}\setminus S$ belongs to the orbit of $
[1:0:0:1:0],$ which implies that $Q^\mathrm{niso}\setminus S$ is a single orbit.
\end{proof}
 Let $G$ be the subgroup of $\mathrm{Aut}(X)$ given by $\mathrm{GL}_2(\C)\rtimes\langle\tilde{\chi}\rangle$, where $\mathrm{GL}_2(\C)$ acts on $X$ as above and $\tilde{\chi}$ is the involution of $X$ discussed previously. Let $E$ be the exceptional divisor of the blow-up $\pi\colon X\to Q^\mathrm{niso}$, let $E^\mathrm{red}$ be its reduced subscheme, and let $\wt{S}$ denote the strict transform of $S$ in $X$.
 
 We now describe the $G$-invariant proper subvarieties of $X$.
\begin{lemma}
\label{lemma:X invariants}
    Then the only $G$-invariant irreducible proper subvariety of $X$ is:
    $$Y:=E^\mathrm{red}\cap \wt{S}.$$
\end{lemma}

\begin{proof}
    Let $Z\subset X $ be a $G$-invariant irreducible proper subvariety; then $\pi(Z)$ is a $\mathrm{GL}_2(\C)$-invariant subvariety of $Q^\mathrm{niso}$ so that $\pi(Z)\subseteq S$ by Lemma \ref{lemma:Q invariants}. This implies that $Z \subseteq \big(\pi^{-1}(S)\big)^\mathrm{red}= E^\mathrm{red}\cup \wt{S}$. Since $Z$ is irreducible and $\langle\tilde{\chi}\rangle$-invariant, we have that $Z\subseteq Y$. It remains to show that $Y$ is integral and forms a single $G$-orbit. For this, note that since $S$ is smooth and the centre of the blow-up satisfies $C_{2,2}^\mathrm{nred}=2C$ as divisors on $S$, and is therefore Cartier, the restriction of $\pi$ to $\wt{S}$ is an isomorphism. It follows that the restriction of $\pi$ to $Y$ is a $\mathrm{GL}_2(\C)$-equivariant isomorphism onto $C$. Since $C$ is a smooth conic on which $\mathrm{GL}_2(\C)$ acts transitively, the lemma follows.
\end{proof}
    We can now complete the proof of K-polystability of $X=\mathrm{Bl}_{C_{2,2}^\mathrm{nred}}(Q^\mathrm{niso})$.

    \begin{proof}[Proof of Lemma \ref{lemma:non-isolated}]
    Let $G\subseteq\mathrm{Aut}(X)$ be the subgroup given above. Let $D$ be a $G$-invariant prime divisor over $X$ and let $Z$ be its centre. By Lemma \ref{lemma:X invariants}, $Z$ is equal to the scheme-theoretic intersection $Y={E^\mathrm{red}}\cap \wt{S}$. Let us estimate $\delta_{Z}(-K_X)$. For this we apply \cite[Theorem 3.3]{AZ22} to the flag ${E^\mathrm{red}}\supset Z$ to get
    $$\delta_Z(-K_X)\geq\mathrm{min}\bigg\{\frac{1}{S(-K_X;{E^\mathrm{red}})},\frac{1}{S(W_{\bullet,\bullet};Z)}\bigg\},$$
    where $W_{\bullet,\bullet}$ is the refinement by ${E^\mathrm{red}}$ of the complete graded linear series associated to $-K_X$. Here we use that fact that $(X,E^\mathrm{red})$ is a plt pair, and also that $Z$ is disjoint from $\mathrm{Sing}(X)$. The latter fact implies that the different $\mathrm{Diff}_{E^\mathrm{red}}(0)$ does not contain $Z$ in its support, and hence that the log discrepancy $A_{E^\mathrm{red},\mathrm{Diff}_{E^\mathrm{red}}(0)} (Z)$ is equal to 1.
    
    First let us compute $S(-K_X;{E^\mathrm{red}})$. Note that  $-K_X-u{E^\mathrm{red}}=(2-u){E^\mathrm{red}}+2\wt{S}$, so that the pseudo-effective threshold of $-K_X$ with respect to $E^\mathrm{red}$ is $u=2$. Then we have the following Zariski decomposition: $-K_X-u{E^\mathrm{red}}=P(u)+N(u)$, where

    \begin{align*}
    P(u)&\sim_\mathbb{R}
    \begin{cases}
 3H-\dfrac{(u+2)}{2}E, & u\in[0,1] \\ 
    \dfrac{3(2-u)}{2}H', & u\in [1,2]
    \end{cases} \\ 
	N(u) &= \begin{cases} 0, & u\in[0,1] \\ 
	2(u-1)\wt{S}, & u\in[1,2]
	\end{cases}
	\end{align*}

    Using the property $S(L; mD)=\frac{1}{m}S(L;D)$ for a $\Q$-Cartier divisor $D$, and the fact that $X$ is a flat degeneration of smooth Fano threefolds in family \textnumero2.21, then we may apply \cite[Lemma 15]{K-STAB} to get:

    $$S(-K_X; E^\mathrm{red})=\dfrac{19}{28}. $$

    Now let us compute $S(W_{\bullet,\bullet};Z)$. The surface $E^\mathrm{red}=\tilde{\chi}(\wt{S})$ is isomorphic to $\F_1$, hence its Picard group is generated by $\mathbf{s}$ and $\mathbf{f}$, where $\mathbf{s}$ denotes the class of the unique $(-1)$-curve on ${E^\mathrm{red}}$ and $\mathbf{f}$ denotes the class of a ruling. Then the class of the intersection $Z=E^\mathrm{red}\cap \wt{S}$ on ${E^\mathrm{red}}$ is $\mathbf{f}+\mathbf{s}$. We have the following Zariski decomposition for the restriction of $-K_X-u{E^\mathrm{red}}$ to ${E^\mathrm{red}}$:

    \begin{align*}
    P(u)|_{E^\mathrm{red}}&\sim_\mathbb{R}
    \begin{cases}
 (4-u)\mathbf{f}+\dfrac{u+2}{2}\mathbf{s}, & u\in[0,1], \\[10pt]
    3(2-u)(\mathbf{f}+\dfrac{\mathbf{s}}{2}), & u\in [1,2],
    \end{cases} \\
	N(u)|_{E^\mathrm{red}} &= \begin{cases} 0, & u\in[0,1], \\ 
	2(u-1)Z, & u\in[1,2],
	\end{cases}
	\end{align*}
and $P(u)|_{E^\mathrm{red}}-vZ$ is given by:

$$
    P(u)|_{E^\mathrm{red}}-vZ\sim_\mathbb{R}
    \begin{cases}
 (4-u-v)\mathbf{f}+(1+\frac{u}{2}-v)\mathbf{s}, & u\in[0,1], \\
    (6-3u-v)\mathbf{f}+(3-\frac{3}{2}u-v)\mathbf{s}, & u\in [1,2]
    \end{cases},$$\\
    which is nef if and only if it is pseudo-effective if and only if $v$ belongs to the interval:
    $$
    \begin{cases}
        \bigg[0,\dfrac{u+2}{2}\bigg], & u\in[0,1], \\[10pt]
        \bigg[0,\dfrac{6-3u}{2}\bigg], & u\in[1,2].
    \end{cases}
    $$
    Thus we compute 

    \begin{align*}
    S(W_{\bullet,\bullet};Z) &= \dfrac{3}{(-K_X)^3}\int_0^2\bigg( (P(u)|_{{E^\mathrm{red}}})^2\mathrm{ord}_ZN(u)|_{{E^\mathrm{red}}}+\int_0^\infty\mathrm{vol}(P(u)|_{{E^\mathrm{red}}}-vZ)dv\bigg)du \\ 
    &= \dfrac{3}{28}\Bigg(\frac{27}{2}\int_1^2(u-1)(u-2)^2du  +\frac{1}{4}\int_0^1\int_0^{\frac{u+2}{2}}(2+u-2v)(14-5u-2v)dvdu \\ &+\frac{1}{4}\int_1^2\int_0^{\frac{6-3u}{2}}(3u+2v-6)(9u+2v-18)dvdu\Bigg)  \\ 
    &= \frac{19}{28}.
    \end{align*}
    Thus, $\delta_Z(-K_X)\geq\dfrac{28}{19}>1$ so that $X$ is K-polystable.
\end{proof}

\section{Deformation theory of complete-intersection threefolds in $\wt{\P}^5$}
\label{section:deformation}
In this section, we complete the proof of Theorem \ref{theorem:main}. It remains only to prove that the space $M^\text{\textnumero 2.21}$ is normal, for which it is enough to prove that every threefold parametrised by $M^\text{\textnumero 2.21}$ has unobstructed deformations. Since every such threefold, other than $\mathrm{Bl}_{C_{2,2}^\mathrm{nred}}(Q^\mathrm{niso})$, has terminal singularities, and hence has unobstructed deformations, it remains to prove this for $\mathrm{Bl}_{C_{2,2}^\mathrm{nred}}(Q^\mathrm{niso})$. We prove the following slightly more general result:

\begin{lemma}
\label{lemma:unobstructed}
    Let $\pi\colon \wt{\P}^5\to \P^5$ denote the blow-up of $\P^5$ along a Veronese surface, with exceptional divisor $E$, and let $X\subset\wt{\P}^5$ be a reduced subscheme of the form $X=H\cap H'$, where  $H\in|\pi^*\O_{\P^5}(1)|$ and $H'\in|\pi^*\O_{\P^5}(2)-E|$. Then deformations of $X$ are unobstructed.
\end{lemma}

\begin{proof}
Since $X$ is a reduced local complete intersection (lci), the statement follows from \cite[Proposition 2.4.8]{DefTheory} provided that the group $\mathrm{Ext}_X^2(\Omega_X^1,\O_X)$ vanishes. By Serre duality, this is equivalent to showing that $H^1(X,\Omega_X^1\otimes\omega_X)=0$ (where we use the fact that $H$ and $H'$ are irreducible and never coincide, which implies that $X$ is lci and hence that the dualising sheaf $\omega_X$ of $X$ is invertible). For this, consider the conormal exact sequence for $X\subset\wt{\P}^5$:
$$0\longrightarrow\mathcal{N}_{X/\wt{\P}^5}^\vee\longrightarrow\Omega^1_{\wt{\P}^5}|_X\longrightarrow\Omega_X^1\longrightarrow 0,$$
where we again use the fact that $X$ is reduced and lci, along with smoothness of $\wt{\P}^5$. By tensoring with $\omega_X$ and taking the long exact sequence of cohomology, then to show that $H^1(X,\Omega_X^1\otimes\omega_X)=0$ it suffices to show that 
$$H^1(X, \Omega^1_{\wt{\P}^5}|_X\otimes\omega_X)=0$$
and
$$H^2(X, \mathcal{N}_{X/\wt{\P}^5}^\vee\otimes\omega_X)=0.$$
For this, consider the Koszul resolution of $\O_X$:
$$0\longrightarrow\O_{\wt{\P}^5}(-H-H')\longrightarrow\O_{\wt{\P}^5}(-H)\oplus\O_{\wt{\P}^5}(-H')\longrightarrow\O_{\wt{\P}^5}\longrightarrow\O_X\longrightarrow0.$$
Notice that $\omega_{\wt{\P}^5}\cong \O_{\wt{\P}^5}(-2H-2H')$ and $\mathcal{N}_{X/\wt{\P}^5}^\vee\cong \O_X(-H|_X)\oplus\O_X(-H'|_X)$, so that by adjunction $\omega_X\cong \O_X(-H|_X-H'|_X)$. To prove that $H^1(X, \Omega^1_{\wt{\P}^5}|_X\otimes\omega_X)$ (resp. $H^2(X, \mathcal{N}_{X/\wt{\P}^5}^\vee\otimes\omega_X)$) vanishes, then by tensoring the Koszul resolution by the locally free sheaf $\Omega^1_{\wt{\P}^5}(-H-H')$ (resp. $\O_{\wt{\P}^5}(-2H-H')\oplus\O_{\wt{\P}^5}(-H-2H') $) and considering the resulting two long exact sequences of cohomology, it is enough to show that
\begin{align*}H^1(\wt{\P}^5, \Omega^1_{\wt{\P}^5}(-H-H'))&=H^2(\wt{\P}^5, \Omega^1_{\wt{\P}^5}(-2H-H'))=H^2(\wt{\P}^5, \Omega^1_{\wt{\P}^5}(-H-2H'))\\
&=H^3(\wt{\P}^5,\Omega^1_{\wt{\P}^5}(-2H-2H'))=0\end{align*}
\begin{align*}
    \Bigg(\text{resp.  }&H^2(\wt{\P}^5, \O_{\wt{\P}^5}(-2H-H'))=H^2(\wt{\P}^5, \O_{\wt{\P}^5}(-H-2H'))=\\
    =&H^3(\wt{\P}^5, \O_{\wt{\P}^5}(-3H-H'))=H^3(\wt{\P}^5, \O_{\wt{\P}^5}(-2H-2H'))=
    H^3(\wt{\P}^5, \O_{\wt{\P}^5}(-H-3H'))=\\
    =&H^4(\wt{\P}^5,\O_{\wt{\P}^5}(-3H-2H'))=H^4(\wt{\P}^5,\O_{\wt{\P}^5}(-2H-3H'))=0\Bigg).
\end{align*}
But both sets of statements are true by Kodaira--Akizuki--Nakano vanishing since $\wt{\P}^5$ is smooth, $H$ and $H'$ are nef, and $H+H'$ is ample. 
\end{proof}

\begin{proof}[Proof of Theorem \ref{theorem:main}]
Let $T\to \P^2$ be the blow-up along the set $\Sigma$. Then by construction of the families $\{\mathscr{X}_{ij}\to U_{ij}\}$ and Lemmas \ref{lemma:smooth}, \ref{lemma:isolated} and \ref{lemma:non-isolated}, we obtain a moduli map
$$T\to M^{\text{\textnumero} 2.21}.$$
By Lemma \ref{lemma:T-orbit-relation}, the fibres of this map are given by $\Gamma$-orbits, where $\Gamma$ is the group described in Definition \ref{def:H}. This implies that the induced morphism
$$T/\Gamma \to M^{\text{\textnumero} 2.21}$$
is bijective, and moreover proper since $T/\Gamma$. Since $M^{\text{\textnumero} 2.21}$ is normal by \cite[Proposition 3]{Namikawa} and Lemma \ref{lemma:unobstructed}, then by Zariski's Main Theorem this map is an isomorphism. 

Under this isomorphism, the locus of negative curves $\Delta_1\cup \Delta_2\cup\Delta_3\subset T$  maps onto the divisors $\Delta_1,\Delta_2,\Delta_3\subset M^{\text{\textnumero} 2.21}$ which are described in the statement of Theorem \ref{theorem:main}.
\end{proof}

\printbibliography
\end{document}